\documentclass[USenglish]{article}
\usepackage{fullpage}
\usepackage{amsmath,amsfonts,amsthm,amssymb}
\usepackage{color,xcolor}
\usepackage{ifpdf}
\usepackage{psfrag}
\usepackage{graphicx,graphics}

\usepackage[small]{caption}
\usepackage{subcaption}

\usepackage{xparse}

\usepackage{authblk}

\usepackage{booktabs}
\usepackage{rotating}
\usepackage{multirow}

\usepackage{url}
\usepackage{hyperref}
\usepackage{todonotes}
\usepackage{ulem} 
\newtheorem{theorem}{Theorem}[section]

\newtheorem{proposition}[theorem]{Proposition}
\newtheorem{remark}[theorem]{Remark}
\newtheorem{definition}[theorem]{Definition}

\usepackage{cancel}

\newcommand{\R}{\mathbb R}

\newcommand{ \PP}{\mathbb P}

\newcommand{\D}{\mathbb D}

\renewcommand{\SS}{\mathcal{S}}

\newcommand{\diff}[1]{{\mathrm{d}{#1}}}

\newcommand{\IdxM} {\mathbb{I}}

\newcommand{\diag}{\mathop{\mathrm{diag}}}

\newcommand{\bbf}{{\mathbf {f}}}

\newcommand{\bn}{{\mathbf {n}}}
\newcommand{\s}{\sqrt{2}}
\newcommand{\is}{\frac{\sqrt{2}}{2}}

\newcommand{\dpar}[2]{\dfrac{\partial #1}{\partial #2}}

\newcommand{\bF}{\mathbf{F}}

\newcommand{\ba}{\mathbf{a}}
\newcommand{\bx}{\mathbf{x}}

\newcommand{\bs}{\boldsymbol}
\renewcommand{\div}{\operatorname{div}}
\newcommand{\est}[1]{\left\langle#1\right\rangle}
\newcommand{\bU}{\mathbf{U}}

\newcommand{\1}{\begin{pmatrix}
                 1\\
                 1
                \end{pmatrix}}
\newcommand{\e}[1]{\mathrm{e^{#1}}}
\newcommand{\VV}{\mathcal{V}}

\newcommand{\n}{\mathrm{n}}

\renewcommand{\vec}[1]{\underline{#1}}
\NewDocumentCommand{\mat}{mo}{%
  \IfValueTF{#2}{%
    \underline{\underline{#1}}{#2}
  }{%
    \underline{\underline{#1}}\,
  }%
}

\newcommand{\tu}{\tilde{u}}

\definecolor{darkspringgreen}{rgb}{0., 0.55, 0.3}
\definecolor{dartmouthgreen}{rgb}{0.05, 0.5, 0.06}
\definecolor{etonblue}{rgb}{0.59, 0.78, 0.64}
\definecolor{airforceblue}{rgb}{0., 0.4, 0.66}
\definecolor{arylideyellow}{rgb}{0.91, 0.84, 0.42}
\definecolor{emerald}{rgb}{0.31, 0.78, 0.47}
\definecolor{uclagold}{rgb}{1.0, 0.7, 0.0}
\definecolor{cadmiumorange}{rgb}{0.93, 0.53, 0.18}

\newcommand{\remi}[1]{\textcolor{black}{{#1}}}
\newcommand{\rev}[1]{\textcolor{black}{#1}} 
\newcommand{\revs}[1]{\textcolor{black}{#1}}  
\newcommand{\revt}[1]{\textcolor{black}{#1}}  

\newcommand{\revb}[1]{\textcolor{red}{#1}}

\hypersetup{
pdftitle={}
pdfauthor={P. \"Offner},
pdfpagemode=UseOutlines,
linkbordercolor=0 0 0,
linkcolor=red,
citecolor=blue,
colorlinks=true,
bookmarks = true
}
\begin{document}
\title{{Analysis of the SBP-SAT Stabilization for Finite Element Methods Part I:   Linear Problems}}

\author[$\dagger$]{R. Abgrall}
\author[$\star \star$]{J. Nordstr\"om}
\author[$ \star$]{P. \"Offner\thanks{Corresponding author: P. \"Offner, philipp.oeffner@math.uzh.ch}}
\author[$\ddag $]{S. Tokareva}
\affil[$\dagger$]{Institute of Mathematics,
University of Zurich, Winterthurerstrasse 190
CH-8057 Zurich, Switzerland}
\affil[$\star \star $]{Department of Mathematics, Computational Mathematics, Link\"oping University, SE-581 83 Link\"oping, Sweden}
\affil[$\star \star$]{Department of Mathematics and Applied Mathematics, University of Johannesburg, P.O. Box 524, Auckland Park 2006, South Africa}
\affil[$\star$]{Institute of Mathematics,
Johannes Gutenberg-Universtiy, Staudingerweg 9,
55099 Mainz, Germany}
\affil[$\ddag $]{Theoretical Division, Applied Mathematics and Plasma Physics Group (T-5), Los Alamos National Laboratory, Los Alamos, NM 87545 USA}

\date{
\today}
\maketitle
\begin{abstract}

In the hyperbolic community,  discontinuous Galerkin (DG) approaches are mainly 
applied when finite element methods are considered.
As the name suggested, the DG framework allows a  discontinuity at the element interfaces, which seems for many researchers a favorable property in case of hyperbolic balance laws.
On the contrary, continuous Galerkin (CG) method \remi{obtained from a straightforward discretisation of the weak form of the PDEs}
 appear to be unsuitable for hyperbolic problems.
 To remedy this issue, 
stabilization terms are usually added and  various formulations can be found in 
the literature. \remi{There exists still the perception that continuous Galerkin methods
are not suited to hyperbolic problems, and the reason of this is the continuity of the approximation.} However, this perception is not true and the  stabilization terms 
\remi{can be removed}, in general, \remi{provided the boundary conditions are suitable.}
 In this paper, we deal with this problem, \remi{and} present a 
different approach. We use  the boundary conditions to stabilize the scheme following 
a procedure that  are frequently used 
in the finite difference community. 
Here, the main idea is to impose the boundary conditions weakly and  \remi{\emph{specific}} boundary operators are constructed  such that they   guarantee stability.
This  approach has already been used in the 
discontinuous Galerkin framework, but here we apply it with  a continuous Galerkin scheme.  No internal dissipation is needed even if   unstructured grids are used. Further, we  point out that we do not need exact integration, it suffices if the quadrature rule and the norm in the differential operator are the same,
\revt{such that  the summation-by-parts (SBP) property is fulfilled  meaning that a discrete Gauss Theorem is valid}.
 This contradicts 
the perception in the hyperbolic community that  stability  issues  for pure Galerkin scheme exist. 
In numerical simulations, we verify our theoretical analysis.


\end{abstract}

\section{Introduction}\label{se:Intro}

In recent years, significant efforts have been made to construct
and develop high-order methods for hyperbolic balance 
laws, and most of the  methods are either based on  
finite difference (FD) or finite element (FE) approaches.
In the FE framework, one favorable, if not the most favorable scheme, 
seems to be the discontinuous Galerkin (DG) method introduced by Reed and Hill \cite{reed1973triangular}
because of its stability properties
\cite{cockburn2012discontinuous, hesthaven2002nodal, chenreview, gassner2013skew}.
Many modern DG formulations are based on summation-by-parts (SBP) operators and the recent stability proofs rely on the SBP property
\cite{gassner2013skew, carpenter2014entropy, chen2017entropy, chan2018discretely, 
kopriva2014energy,  ranocha2016summation}.
\rev{Even, if the SBP operators where originally \revt{defined in} the FE framework, they have been transferred
to FD methods \cite{kreiss1974finite} and have been further developed in the FD setting
where they are now commonly used. They lead  to stability following the steps of the continuous energy analysis  \cite{fernandez2014review, hicken2016multidimensional, svard2014review}.}
Together with SBP operators, Simultaneous Approximation Terms (SATs) 
that impose the boundary conditions weakly are   applied. 
The SBP-SAT technique is  powerful and universally applicable as we will show in this paper. 
Another  reason for the popularity of DG is that the numerical solution is allowed to have a discontinuity at the element boundaries,
and, since non-linear hyperbolic problems are supporting shocks, this property is believed to be 
desirable. In addition and maybe  most important, the DG methods leads to block diagonal mass matrices which are easy to invert. 
The difference between a DG approach and continuous Galerkin (CG), besides the structure of the mass matrix, is that in CG 
the approximated solution is forced to be continuous also over the element boundaries.
This restriction  is perceived to be quite strong also in terms of stability
where the erroneous (as we will show) belief in the  hyperbolic research community exists,
that a pure CG scheme is unstable\footnote{We like to mention that also parts of the authors 
had this belief before starting the project.},
and   \remi{stabilization terms have to be  applied to remedy this issue}
\cite{abgrall2019high}.\\
One may only speculate where this erroneous perception come from?
In our opinion, one major reason could be that 
if one considers a pure Galerkin method  using a linear Lagrange polynomial basis of order one,
it can be shown that the method is equivalent to the 3-point central difference scheme.
This scheme is  not von Neumann stable \cite{zbMATH06207795} when periodic boundary conditions are considered. 
By switching the basis functions for instance to splines and / or  
some lumping technique, von Neumann stability can be proven, see
\cite{thomee1974convergence, mock1976explicit}. We want to point out that with the lumping technique as described in 
\cite{mock1976explicit}, one is able to re-write the Galerkin method to  well-known finite difference schemes   like  Lax-Friedrichs or Lax-Wendroff schemes and at the end, a stable  finite difference scheme\footnote{Actually, the opposite is also true: Everything is a finite element scheme with a suitable quadrature rule.}. 
In addition, if one considers initial-boundary value problems,
there also exist some preliminary stability results  \cite{layton1983stable, layton1983stable1, gunzburger1977stability}. Here, the main idea is to switch the norms of the trial space and include the procedure at the boundary.  However, these results seems forgotten in the hyperbolic community.  \\
In this paper, we  focus on the stability property of a pure Galerkin scheme, but follow 
a different approach. Our preliminary idea/thought is:
If one considers the DG method with one element,
the method is stable. 
There is nothing that says that the approximation space must be a broken polynomial space, 
the only thing that is needed is that the trial and test function must have some kind of regularity within the elements,
so that the divergence theorem (or SBP techniques) can be applied. 
\rev{Continuity at the boundaries and regularity inside the elements due to the polynomial space are enough.
No internal  artificial  dissipation is required} \revs{and no 
special conditions on the grid structure, for instance cartesian grids, have to be assumed.  Thus, for example  unstructured triangular  meshes
can be applied.}
Hence, one can see a CG method as a DG one, with only one element 
(the union of the simpleces) with an approximation space made of polynomials with continuity requirement between the simpleces.
Hence, what is the difference   between these two approaches? 
The answer  to this question points to the procedure at the boundary. 
In the stability proofs, the use of SATs is essential. 
In  \cite{hicken2016multidimensional} 
diagonal norm stable CG SBP-SAT discretizations have previously been presented and
further extended in \cite{hicken2019entropy, hicken2020entropy} where local projection stabilizations are applied to obtain 
entropy stable discretizations. The focus lies especially on the construction and investigation of diagonal norm SBP operators.
Contrarily, in this work we focus on SAT \revt{and apply Galerkin schemes which fulfill the SBP property meaning that a discrete Gauss theorem is valid}.  We apply them with pure 
CG discretizations with dense norms and this is the topic of this paper
where we show that no internal dissipation is needed in CG methods.
We divide the paper as follows:
In the second section, we  shorty introduce the continuous Galerkin scheme which is used and investigated in the following.
Next, we introduce and repeat the main idea of the SAT procedure from the FD framework 
and extend it to the Galerkin approach. We show that
 the determination of the boundary operators is essential. In section \ref{sec:boundary_operator},
we focus on the eigenvalue analysis of the spatial operators and derive  conditions
from the continuous setting to build adequate boundary operators in the discrete framework. 
We give some recipes which will be used in section \ref{sec:Numerics} to support our analysis in numerical experiments. Finally, we conclude  and discuss future work.


%
%

\section{Continuous Galerkin  Scheme}\label{eq:Finite_element}

In this section, we shortly introduce the pure continuous Galerkin scheme (CG)
as it is also known in the \remi{literature
\cite{kreiss1974finite, Hughes1, burman2010explicit,   burman2004edge}.}
We are interested in the numerical approximation of a hyperbolic problem 
\begin{equation}\label{eq:conservation_law_general}
 \frac{\partial U}{\partial t}+\div f(U)=0
\end{equation}
with suitable initial and boundary conditions. 
The domain $\Omega$ is split into subdomains $\Omega_h$ 
(e.g triangles/quads in two dimensions, tetrahedrons/hex in 3D). 
We denote by $K$ the generic element of the mesh and by $h$ the characteristic mesh size. 
Then, the degrees of freedom (DoFs) $\sigma$  are defined 
in each $K$:   we have  a set of linear forms acting on the set
$\PP^k$ of polynomials of degree $k$ such that the linear mapping 
$q\in \PP^k\longmapsto (\sigma_1(q),\cdots, \sigma_{|\sum_K|}(q))$ is one-to-one,
\rev{ where $|\sum_K|$ denotes the total number of DoFs in $K$}. The set 
$\SS$ denote the set of degrees of freedom in all elements. 
The solution $U$ will be approximated by some element from the space $\VV^h$
defined by 
\begin{equation}\label{eq:solution_space}
\VV^h:=\rev{\bigoplus_{K} \left\{ U^h|_K \in \PP^k  \cap C^0(\Omega) \right\} .}
\end{equation}
A linear combination of basis functions $\varphi_\sigma\in \VV^h$ 
will be used to describe the numerical solution
\begin{equation}\label{eq:solution_approx}
U^h( x)=\sum_{K\in \Omega_h }\sum_{\sigma \in K} U_{\sigma}^h(t)  \varphi_{\sigma}|_K( x), \quad \forall{ x \in \Omega}
\end{equation}
 As basis functions we are working either with Lagrange interpolation 
where the degrees of freedom are associated to points in $K$  or   B\'{e}zier polynomials. \\
To start the discretisation, we
apply a Galerkin approach and multiply with a test function $V^h$  and integrate over the domain.
This gives
\begin{equation}\label{eq:galerkin7}
 \int_{\Omega} (V^h)^T \frac{\partial U}{\partial t}\diff x+ \int_{\Omega} (V^h)^T \div f(U) \diff{x} =0.
\end{equation}
Using the divergence theorem, we get
\begin{equation}\label{eq:galerkin}
 \int_{\Omega} (V^h)^T \frac{\partial U}{\partial t}\diff x- \int_{\Omega} (\nabla V^h)^T f(U) \diff{x} +\int_{\partial \Omega} \big (V^h\big )^Tf(U)\cdot \n \;d\gamma =0.
\end{equation}
By choosing $V^h=\varphi_{\sigma}$ for any $\sigma \in \SS$,  where we further assume 
for simplicity that 
our basis functions vanishes at the physical boundaries, we obtain with  \eqref{eq:solution_approx} a system of equations: 
\begin{equation}\label{zob}
 \sum_{K\in \Omega_h }\sum_{\sigma' \in K} \bigg (
 \frac{\partial U_{\sigma'}^h(t)}{\partial t} \int_K\varphi_{\sigma'}(x) \varphi_\sigma(x) \diff x
- \int_{K}  \nabla\varphi_{\sigma'}(x)\; f(U^h) \diff{x}\bigg )  =0,
\end{equation}
Our approach makes \eqref{zob} fully explicit with special focus on the basis functions.
The contributions from the boundary integral are removed due to the assumption
that $\varphi|_{\partial \Omega}=0$. 
Our motivation for this simplification is driven by the fact that we want to avoid 
a discussion about boundary conditions in this part which will be the content of the followings sections.  The internal contributions cancel out through 
the different signs in the normals. 
In practice, we compute  \eqref{zob} with a quadrature rule:
\begin{equation*}
 \sum_{K\in \Omega_h }\sum_{\sigma' \in K} \bigg (
 \frac{\partial U_{\sigma'}^h(t)}{\partial t} \; \oint_K\varphi_{\sigma'}(x) \varphi_\sigma(x) \diff x
- \oint_{K}  \nabla\varphi_{\sigma'}(x)\; f(U^h) \diff{x} \bigg ) =0,
\end{equation*}
where $\oint$ represents the quadrature rules for the volume and surface integrals. 

 In this paper,  we are considering  linear problems, i.e. the flux is linear in $U$, but may depend on the spatial coordinate. In all the numerical experiences, we will make the spatial dependency simple enough (i.e. typically polynomial in $x$), so that it will always be possible to find a standard quadrature formula and obtain accurate 
 approximations for the integrals. Note, if the quadrature rule is accurate enough, \eqref{zob} can be 
 exactly reproduced for linear problems with constant coefficients.
  
 

Using a matrix formulation,
we obtain the classical FE formulation:
 \begin{equation}\label{eq:finite_semidiscrete}
  \mat{M}\frac{\partial}{\partial t}\vec{U}^h +\mat{\bF}=0
 \end{equation}
 where $\vec{U}^h$ denotes the vector of degrees of freedom, $\mat{\bF}$ 
 is the approximation of $\div f$  and 
 $\mat{M}$ is a mass matrix
 \footnote{In the finite difference community $\mat{M}$ is called norm matrix and is classically 
 abbreviated with $P$, c.f. 
 \cite{svard2014review, nordstrom2006conservative}.}.
For  continuous elements, this matrix is sparse but not block diagonal, contrary to the situation for the discontinuous
 Galerkin methods.
Due to the rumor/perception in the hyperbolic community that a pure Galerkin scheme suffers from stability issues
for hyperbolic problems, \remi{it is common to add  stabilization 
 terms to the scheme as for example in \cite{burman2004edge}, and this is a very safe solution.}
 However,  as previously mentioned we take a different approach in this paper
 and will renounce these classical stabilization techniques. 
 In order to do this,  we  need more known results  from the literature,
 which we will briefly repeat here.

\section{Weak Boundary Conditions}\label{sec:SAT}

To preserve the structure of the SBP operators, and facilitate proofs of stability, weak boundary conditions are preferable over strong one's. 
\subsection{SATs in SBP-FD framework}\label{subsec:SAT_FD}
To implement the boundary conditions weakly using simultaneous approximation terms (SATs) 
is nowadays standard in the FD community and has  been developed there.
Together with summation-by-parts (SBP) operators it provides a powerful tool for proofs of semidiscrete ($L_2$) stability of linear problems by the energy method, see
\cite{fernandez2014review, svard2014review, nordstrom2017roadmap} for details. 

Here, we present a short introductory example of the SBP-SAT technique as it is presented in \cite{nordstrom2006conservative, svard2014review}.
Consider the linear advection equation 
 \begin{equation}\label{eq:linear_ad}
 \begin{split}
    \frac{\partial u}{\partial t}+a \frac{\partial u}{\partial x}&=0, \quad 0\leq x\leq 1, \quad t>0,\\
    u(x,0)&=u_{in}(x),\\
    u(x,t)&=b(x,t)\text{ for  inflow boundary}
 \end{split}
 \end{equation}
where $u_{in} \in $ is the initial condition and $b$ is the \rev{known} boundary data that is only defined on the inflow part of $\partial[0,1]=\{0,1\}$. In other words, if $a>0$, then $b$ is only set for $x=0$, and if $a<0$, this will be for $x=1$ only. \\
To explain the semi-discretisation of \eqref{eq:linear_ad}, 
we consider the discrete grid $\vec{x}=(x_0, \cdots, x_N)^T$, with the ordering of nodes
$x_0=0<\cdots<x_N=1$. Furthermore, the spatial derivative of a function $\phi$ is approximated through  
a discrete derivate matrix $\mat{D}$, i.e. $\phi_x\approx \mat{D} \vec{\phi}$ with 
$\vec{\phi}=(\phi(x_0),\; u_1,\;, \cdots, \phi(x_N))^T$. It is defined by 
\begin{definition}[SBP operators]\label{de:SBP}
An operator \mat{D} is a $p$-th order accurate approximation of the first derivative on SBP form 
if 
\begin{enumerate}
\item $\mat{D}\vec{x}^j=\mat{M}^{-1}\mat{Q} \vec{x}^j= j \vec{x}^{j-1}, \; j \in [0,p] $ with $\vec{x}^j= (x_0^j, \cdots, x_N^j)^T$,
\item $\mat{M}$ is a symmetric positive definite matrix, 
 \item $\mat{Q}+ \mat{Q}^T=\mat{B}=\diag(-1,0, \cdots,0,1)$ holds. 
\end{enumerate}
\end{definition}
Now, a semi-discretisation of \eqref{eq:linear_ad}
is given in terms of SBP operators as
 \begin{equation}\label{eq:linear_ad:semidisc}
 \begin{split}
    \frac{\partial  \vec{u}}{\partial t}+a\mat{D}\vec{u}&=\mat{M}^{-1}\vec{\mathbb{S}} , \quad t>0,\\
    \vec{u}(0)&=\vec{u}_{in},
 \end{split}
 \end{equation}
where $\vec{u}=(u_0,\; u_1,\;, \cdots, u_N(t))^T$ are the coefficients of $u$ and similarly for $\vec{u}_{in}$. The coefficients are evaluated on the nodal values, i.e. the grid points, and are used to express the numerical solution \eqref{eq:solution_approx}.
Translating this into the FE framework, they correspond to the coefficients for the degrees of freedoms.
The symmetric positive definite matrix $\mat{M}$ approximates the usual $L^2$ scalar product. Together with condition $3.$ from Definition \ref{de:SBP},  we mimic integration by parts  discretely, i.e. 
\begin{equation}\label{eq:SBP_property}
 v(1)u(1)-u(0)v(0)=\int_0^1 u(x) v'(x) \diff x  +  \int_0^1 u'(x) v(x) \diff x \approx \vec{u} \mat{M}\mat{D}\vec{v}  +\vec{u} \mat{D}^T\mat{M} \vec{v}=\vec{u}  \mat{B}\vec{v} 
\end{equation}
In \eqref{eq:SBP_property}, we have  for smooth functions $u$ 
\begin{equation}\label{eq:approx_matrix}
 \mat{D}\vec{u}\approx   \frac{\partial}{\partial x} u \text{ and }
||\vec{u}||_{\mat{M}}^2:= \vec{u}^T\mat{M}\vec{u} \approx \int_{0}^{1} u^2(x)\diff x.
\end{equation}

Instead of having an extra equation on the boundary like in \eqref{eq:linear_ad}, the boundary condition 
is enforced weakly by the term 
$\vec{\mathbb{S}}=(\mathbb{S}_0,0,\cdots, \mathbb{S}_N)^T$ which is called the SAT. 
We demonstrate how it should be selected to guarantee stability for \eqref{eq:linear_ad:semidisc}.

\begin{definition}\label{def:energy_stable}
 The scheme \eqref{eq:linear_ad:semidisc} is called  strongly energy stable  if 
 \begin{equation}\label{es:energy_stable}
  ||\vec{u}(t)||_{\mat{M}}^2\leq K(t)\left(||\vec{u}_{in}||_{\mat{M}}^2+
  \max_{t_1\in [0,t]} \remi{|b(t_1)|}^2  \right)
 \end{equation}
 holds. 
 The term $K(t)$ is bounded for any finite $t$  and independent from $u_{in}$, $b$ and the mesh. 
\end{definition}

\begin{remark}
The Definition \ref{def:energy_stable} is formulated in terms of the initial value problem \eqref{eq:linear_ad} where only one 
boundary term is fixed. If an 
additional forcing function is considered at the right hand side of
\eqref{eq:linear_ad}, we include the maximum of this function
in \eqref{es:energy_stable} in the spirit of $b$,  for details see \cite{svard2014review}. 
\end{remark}

As established in \cite{carpenter1994time}, we can prove the following:
\begin{proposition}\label{prop:energy_stable}
Let $\mat{D}=\mat{M}^{-1}\mat{Q}$ be an SBP operator defined in   \ref{de:SBP} with $\mat{Q}$ fulfilling 
\begin{equation}\label{eq:Q_import}
 \mat{Q}+\mat{Q}^T=\mat{B}=\diag(-1,0, \cdots,0,1).
\end{equation}

Let $a^+=\max\{a,0\}$ and $a^-=\min\{a,0\}$, $b_0=b(0,t)$ and $b_N=b(1,t)$. If $\mathbb{S}_0=\tau_0 a^+ (u_0-b_0)$  and $\mathbb{S}_N=\tau a_N^- (u_N-b_N)$  with $  \tau_0, \tau_N < - \frac{1}{2}$,
then the scheme \eqref{eq:linear_ad:semidisc} is \remi{strongly} energy stable. 
\end{proposition}
\begin{proof}
 Multiplying \eqref{eq:linear_ad:semidisc} with $\vec{u}^T\mat{M}$ yields 
 \begin{equation}\label{eq:SBp}
  \vec{u}^T\mat{M}  \frac{\partial}{\partial t} \vec{u}+a\vec{u}^T\mat{M}\mat{D}\vec{u}=\vec{u}^T\vec{\mathbb{S}}. 
 \end{equation}
 Transposing \eqref{eq:SBp} and adding both equations together leads to 
 \begin{equation*}
\frac{\rm{d}}{\rm d t } ||\vec{u}||_{\mat{M}}^2= \vec{u}^T\mat{M}   \frac{\partial}{\partial t}  \vec{u} +   \frac{\partial}{\partial t}\vec{u}^T \mat{M}\vec{u}=-a \vec{u}^T (\mat{Q}+\mat{Q}^T) \vec{u}+2\vec{u}^T\vec{\mathbb{S}}.
 \end{equation*}
Further, we obtain from \eqref{eq:Q_import}
 \begin{equation*}
 \frac{\rm{d}}{\rm d t }  ||\vec{u}||_{\mat{M}}^2= \bigg (a u_0^2 +2a^+\tau u_0(u_0-b_0)\bigg ) -\bigg (au_{N}^2 -2a^-\tau u_N(u_N-b_N)\bigg ).
 \end{equation*}
 If  $\tau_0, \tau_N<- \frac{1}{2}$, we find
 \begin{equation*}
  \frac{\rm{d}}{\rm d t }  ||\vec{u}||_{\mat{M}}^2\leq -\frac{a^+\tau^2}{(1+2\tau)} b_0^2+\frac{a^-\tau^2}{(1+2\tau)} b_N^2. 
 \end{equation*}
\end{proof}
This shows that the boundary operator $\mathbb{S}$ can be 
chosen in such way that it guarantees stability for the SBP-SAT approximation of \eqref{eq:linear_ad}.  
Next, we will apply this technique in the Galerkin framework.

\subsection{SATs in the Galerkin-Framework}\label{subsec:SATs:FE}
Instead of working with SBP-FD framework we consider now   a Galerkin approach for the approximation
of \eqref{eq:linear_ad}. 
In \cite{carpenter2014entropy, gassner2013skew},  it is shown that the specific DG schemes 
satisfies a discrete summation-by-parts (SBP) property and can be interpreted as SBP-SAT schemes
with a diagonal mass matrix. In this context, one speaks about the \textbf{discontinuous Galerkin 
spectral element method} (DGSEM). Here, we consider nodal continuous Galerkin methods and 
focus on stability conditions in this context.
As we described already in Section \ref{eq:Finite_element}, the
difference between the continuous and discontinuous Galerkin approach is the 
solution space \eqref{eq:solution_space} and the  structure of the mass matrix \eqref{eq:finite_semidiscrete}
which is not block diagonal in CG. 
\revt{However, in the following we consider only Galerkin approaches which fulfill the SBP property meaning that 
a discrete Gauss theorem is valid. }
The approach with SAT 
terms can still be used to ensure stability 
also in case of CG but one has to be precise, as we will explain
in the following.
Let us step back to the proof of Proposition \ref{prop:energy_stable} and have a
closer look. Essential in the proof is  condition \eqref{eq:Q_import}. 
Let us focus on this condition for a \remi{Galerkin} discretisation of \eqref{eq:linear_ad} as described also in \cite{nordstrom2006conservative}.
We approximate \rev{equation \eqref{eq:linear_ad}} now with  $u^h(x, t)=\sum\limits_{j=0}^N u^h_j(t)\varphi_j(x)$ where $\varphi_j$ are basis functions and $u^h_j$ are the coefficients. 
First, we consider the problem without including the boundary conditions.  
Let us assume that $\varphi_j$ are Lagrange polynomials where the degrees of freedoms are associated to points in the interval.
Introducing the scalar product
$$\est{u,v}=\int_I u(x)v(x)\; \diff x ,$$  
let us consider the variational formulation of the advection equation \eqref{eq:linear_ad}  with test function $\varphi_i$. We insert   the approximation and get
\begin{equation*}
\begin{split}
\est{   \frac{\partial}{\partial t} u^h(t,x), \varphi_i(x) }+ \est{a  \frac{\partial}{\partial x}u^h(t,x), \varphi_i (x)}&{=} 0, \quad \forall i=0, \cdots, N,\end{split}\end{equation*}
i.e.
\begin{equation*}\begin{split}
 \int_I \sum_{j=0}^N (  \frac{\partial}{\partial t} u^h_j(t)) \varphi_j(x)\varphi_i(x) \diff x+a \int_I \sum_{j=0}^Nu^h_j(t)(  \frac{\partial}{\partial x}\varphi_j(x) ) \varphi_i(x) \diff x&=0.
\end{split}\end{equation*}
Finally, we get 
\begin{equation}\label{eq:approx}\begin{split}
 \sum_{j=0}^{N} M_{i,j}(  \frac{\partial}{\partial t}u^h_j(t) ) + a\sum_{j=0}^{N} Q_{i,j} u_j^h(t)&=0\\
\end{split}
\end{equation}
with
\begin{equation}\label{eq:definition_norm_matrix}
M_{i,j}= \int_I  \varphi_j(x)\varphi_i(x) \diff x \quad \text{ and } \quad Q_{i,j}= \int_I (  \frac{\partial}{\partial x} \varphi_j(x)) \varphi_i(x) \diff x.
\end{equation}
In matrix formulation \eqref{eq:approx} can be written 
\begin{equation*}
 \mat{M}\frac{\partial}{\partial t} \vec{u} +a\mat{Q}\vec{u}=0
\end{equation*}
as  described in \cite{nordstrom2006conservative}.
Let us check \eqref{eq:Q_import}.  We 
consider
\begin{equation}\label{eq:Q_inte}
\begin{split}
Q_{i,j}+Q_{i,j}^T&= \int_I (  \frac{\partial}{\partial x} \varphi_j(x)) \varphi_i(x) \diff x +\int_I (  \frac{\partial}{\partial x} \varphi_i(x)) \varphi_j(x) \diff x
= \int_I  \frac{\partial}{\partial x} \left(  \varphi_j(x) \varphi_i(x) \right) \diff x\\
&= \varphi_i(x)\varphi_j(x)|_{0}^1=\varphi_i(1)\varphi_j(1)-\varphi_i(0)\varphi_j(0) \quad \forall i,j =0, \cdots, N.
\end{split}
\end{equation}
If the boundaries are included in the set of degrees of freedom,
then we obtain  
\begin{equation*}
\varphi_i(1)\varphi_j(1)-\varphi_i(0)\varphi_j(0)=\begin{cases}
									1 &\text{ for } i=j=N,\\
									-1 &\text{ for } i=j=0,\\
									0 &\text { elsewhere}.
									\end{cases}
\end{equation*}
Up to this point exact integrals are considered but the same steps are valid if  a quadrature rule is applied
such that   \eqref{eq:Q_import} is satisfied and \eqref{eq:Q_inte} is mimicked on the discrete level. 
\revt{This is ensured if the SBP property is fulfilled. Note that in this paper we only consider Galerkin schemes which guarantee this property. }
However, if we  include a weak boundary condition similar to \eqref{eq:linear_ad:semidisc},
we obtain the semidiscrete scheme 
\begin{equation}\label{eq:approx_SAT}\begin{split}
 \sum_{j=0}^{N} M_{i,j}(  \frac{\partial}{\partial t}u^h_j(t) ) + a\sum_{j=0}^{N} Q_{i,j} u_j^h(t)&=\mathbb{S}\\
\end{split}
\end{equation}
with the SAT term given by 
\begin{equation}\label{eq:SAT_GAlerking}
 \mathbb{S}:=\tau a^+ (u_0-b_0)\delta_{x=0}+\tau a^- (u_N-b_N)\delta_{x=x_N=1}.
 \end{equation}
By following the steps from the proof of Proposition  \ref{prop:energy_stable},
we can prove:

\begin{proposition}\label{prop:galerkin_energy}
If the Galerkin method  \eqref{eq:approx_SAT}  is applied to solve \eqref{eq:linear_ad}
with $\mathbb{S}$ given by \eqref{eq:SAT_GAlerking}  and  $\tau \remi{<} - \frac{1}{2}$,
it  is \remi{strongly} energy stable. 
\end{proposition}
\begin{proof}
\rev{The weak formulation of the problem reads:
\begin{equation*}
\begin{split}
\est{\frac{\partial}{\partial t} u^h(t,x), \varphi_i(x) }+ \est{a  \frac{\partial}{\partial x}u^h(t,x), \varphi_i (x)}&{=} 
\tau a^+ (u^h(0,t)-b_0(t))\varphi_i(0)
+\tau a^- (u^h(1,t)-b_N)\varphi_i(1), \end{split}\end{equation*}
for all $i=0,\cdots, N$, where  for simplicity, we consider the case $a>0$.}
The SAT techniques adds a penalty term into the approximation \eqref{eq:approx_SAT} on the right side. 
We focus now on the energy. Therefore, we multiply also with $u^h$ instead of $\varphi_i$ and rearrange the terms.
 We obtain for the semi-discretization  \eqref{eq:approx_SAT}:
\begin{equation*}
  \sum_{i,j,=0}^{N} M_{i,j}\left(  \frac{\partial}{\partial t}u^h_j(t) \right) u^h_i(t) + a\sum_{i,j=0}^{N} Q_{i,j} u_j^h(t)u_i^h(t) = a\tau u^h_0(t)(u^h_0(t)-b_0(t))
\end{equation*}
where we used the fact that $u^h(t,0)=\sum_{i=0}^N u_i^h(t)\varphi_i(0)=u_0^h(t) $ is valid. 
By following  the steps of the proof of Proposition \ref{prop:energy_stable} and using 
\eqref{eq:Q_inte} we get the final result. 
\end{proof}
In the derivation above, we restricted ourselves to
one-dimensional problems using Lagrange interpolations. 
Nevertheless, this shows that 
a continuous Galerkin method is stable if 
the boundary condition is enforced by a proper penalty term. 
For the general FE semi-discretization of
\eqref{eq:finite_semidiscrete},  the procedure is similar and straightforward.
Without loss of generality, it is enough to consider homogeneous boundary conditions and 
for a general   linear problem (scalar or systems) the formulation \eqref{eq:finite_semidiscrete}
can be written with penalty terms as 
\begin{equation}\label{eq:FD_advection}
\mat{M} \frac{\partial} {\partial t}\vec{\bf U}^h+\mat{Q_1}\ \vec{\bU}^h= \rev{\Pi} ( \vec{\bU}^h )
\end{equation}
where $ \rev{\Pi} $ is the boundary operator 
which includes the boundary conditions. This operator can be expressed  in the discretization by a matrix vector multiplication. 
With a  slight of abuse of notation,  we denote this boundary matrix with $\mat{\Pi}$ and it  is usually sparse. Further, 
$\mat{Q_1}$  represent the spatial operator and $\mat{Q_1} +\mat{Q_1}^T $ has only contributions on the boundaries. 
Then, we can prove.
 
\begin{theorem} \label {general}
Apply  the general FE semidiscretisation \eqref{eq:FD_advection}  together with the SAT approach 
to a linear equation
and let the mass matrix $\mat{M}$  of   \eqref{eq:FD_advection} be symmetric and positive definite. 
If the boundary operator $\mat{\Pi}$ together with the discretization $\mat{Q_1}$ can be chosen such that 
\begin{equation}\label{test}
(\mat{\Pi}-\mat{Q_1})+ \left(\mat{\Pi}-\mat{Q_1} \right)^T
\end{equation}
\rev{is negative semi-definite, then the scheme is energy stable. }
\end{theorem}
\begin{proof}
We use the energy approach and multiply our discretization with $\bU^h$ instead of $\varphi_i$
and add the transposed equation using $\mat{M}^T=\mat{M}$. 
We obtain 
\begin{equation*}
\dfrac{d} {d t} || \vec{\bU}^h||^2_{\mat{M}} =\vec{\bU}^{h,T} \left(
(\mat{\Pi}-\mat{Q_1})+ \left(\mat{\Pi}-\mat{Q_1}\right)^T \right)
\vec{\bU}^{h} 
 \leq0.
\end{equation*}
\end{proof}

\begin{remark}\label{re:exactness}
This theorem yields  directly  conditions \remi{when a FE method is stable, or not.}
If the matrix \eqref{test} has positive eigenvalues $\{\lambda_i\} $, stabilization terms have to be added. 
\textbf{However, no internal stabilization terms are necessary when  \eqref{test}  is negative semi-definite.}
Therefore,  a number of requirements are needed. 
The \revs{ distribution of the residuals attached to the degrees of freedom should be  done  in an "intelligent" way e.g.  if we consider triangle elements and polynomial order $p=1$, we set the DoFs in every edge and not all of them on one face.}
Further,  the chosen  quadrature rule  in the  numerical integration has to be the same as in the differential operators. This means that  the applied quadrature rule to  calculate  the mass matrix should be the same as  
the used one to determine  the differential operators\revt{,  such that SBP property is fulfilled meaning that a discrete Gauss Theorem is valid.}
In the numerical test, we will present an example of what happens if the chosen quadrature rules 
disregard this.
Furthermore, in case of a non-conservative formulation of the hyperbolic problem or in case of variable 
coefficients  
a skew-symmetric/split  formulation should be applied in the way described in \cite{nordstrom2006conservative, ranocha2017extended, offner2019error}.
In the one-dimensional setting, we obtain in the continuous case
\begin{equation*}
\partial_x (au) =\alpha \partial_x(au)  + (1-\alpha)  \left( u (\partial_x a) + a(\partial_x u)\right) 
\end{equation*}
with $\alpha=0.5$ and the implementation has to mimic this behavior.\\
If the implementation of the continuous Galerkin method is done in such way that 
the matrix  \eqref{test}  is negative semi-definite, then  the method is stable only through our  boundary procedure.
In our opinion,  this is \textbf{contrary} to common belief about continuous Galerkin methods for hyperbolic problems.
The only stabilizing factor needed is a proper implementation of boundary conditions. 
For the linear scalar case, the proof is given in Proposition 
\ref{prop:galerkin_energy}.
In the following, we will extend this theory to more general 
cases.
\end{remark}

\begin{remark}[Weak Boundary  Conditions in Galerkin Methods for Hyperbolic Problems]\label{re:Further_similar}
The weak formulation of the boundary condition is not done for the first time to analyze 
stability properties in continuous Galerkin  methods. As already mentioned in the introduction, in 
\cite{layton1983stable, layton1983stable1, gunzburger1977stability} the authors have 
included the procedure at the boundary in their stability analysis where the main idea is to 
switch the norm of the trial space to prove stability.  \\
Further, in  \cite{bazilevs2007weak} the authors have compared weak and strong implementation of the boundary conditions 
for incompressible Navier Stokes when the boundary condition is discontinuous
and $C^0$ approximations are used. Here, non-physical oscillations are arising and  by switching to the weak implementation, the authors have been able solve this issue.
However, as a baseline schemes they have always supposed a stabilized Galerkin methods like SUPG 
in their theoretical considerations  and applied it  in their numerical examples.
They have not imposed the weak boundary condition to stabilize their baseline scheme, but to cancel out these 
oscillations. In  Nitsche's method \cite{nitsche1971variationsprinzip} for elliptic and parabolic problems, 
the boundary conditions are also imposed weakly. Here, the theoretical analysis is based on the bilinear 
from. However, further  extensions of this method can be found and also a comparison 
to several DG methods. For an introduction and some historical remark, we strongly recommend 
 \cite{cockburn2012discontinuous} for more details.
 Finally, in DG methods it is common to impose boundary conditions weakly in hyperbolic problems  and a detailed analysis 
 for Friedrichs' system \cite{friedrichs1958symmetric} can be found in \cite{ern2006discontinuous} which is oriented more on the variational formulation. \\
 Finally, we want to point out again that the purpose of this paper is to demonstrate that no further internal dissipation is needed 
 if the boundary conditions are implemented correctly. In addition, our analysis holds also if we apply unstructured 
 grids as demonstrated in the numerics section below. 
\end{remark}

\section{Estimation of the SAT-Boundary Operator}\label{sec:boundary_operator}

As described before, a proper implementation of the boundary condition  
is essential for stability. 
Here, we give a  recipe for  how these SAT boundary operators can be chosen to
get a stable CG scheme for different types of problems.
First, we consider a scalar equation in 2D and 
transfer the eigenvalue analysis for the spatial operator from 
the continuous to the discrete setting. 
Then, we extend our investigation to two dimensional systems.
Using again the continuous setting, we develop estimates 
for $\Pi$ and transfer the results to the finite element framework.
We apply them later in the numerical section.

\subsection{Eigenvalue Analysis}\label{section_eigenvalue}
We derive  conditions on the boundary 
operators and perform an  eigenvalue analysis to get \remi{an energy estimate} in the continuous setting.
Next, the results are transformed to the 
discrete framework to guarantee stability of the discrete scheme. 

\subsubsection{The scalar case}
\subsubsection*{Continuous Setting}

Consider the initial boundary value problem 
\begin{equation}\label{eq:continous_eig}
\begin{aligned}
\frac{\partial} {\partial t}u+a \frac{\partial} {\partial x} u+b \frac{\partial} {\partial y}u &=0 \quad &x\in \Omega, \quad &t>0, \\
Bu&=g \quad &x\in \partial \Omega, \quad &t>0,\\
u&=f  \quad &x\in \Omega,\quad &t=0
\end{aligned}
\end{equation}
 in the spatial domain $\Omega \subset \R^2$.
Further, $a,b \in  \R$, the function $f$ describes the initial condition, $B$ represents the boundary operator and 
the function $g$ the boundary data.
Without loss of generality,  it is enough to  consider  
homogeneous boundary conditions and we consider the spatial operator
\begin{equation}\label{eq:spatial}
\begin{aligned}
Du&:= \left(a\frac{\partial}{\partial x} +b\frac{\partial}{\partial y}\right) u, \\
\end{aligned}
\end{equation}
considered in the subspace of functions for which $Bu=0$.
This operator will be dissipative  if $\est{u,Du}>0$. Using the 
 Gauss-Green theorem, we
obtain 
\begin{equation}
\est{u,Du}=\int_\Omega uDu \; \diff{\Omega}=
\int_{\partial \Omega}\frac{a}{2}u^2 \diff{y} -\frac{b}{2}u^2\diff{x}   
= \remi{\frac12} \int_{\partial \Omega} \underbrace{ (a,b) \cdot \n}_{:=a_n} u^2 \diff{s}.
\end{equation}
The operator is hence dissipative if 
$
\int_{\partial \Omega} a_n u^2\diff{s}>0.
$ 
The question rises: How do we guarantee this condition? This is the role of the boundary conditions, i.e. when $a_n\leq 0$, we need to impose 
 $u=0$.  For outflow, i.e. $\partial \Omega_{out} $ 
we have $a_n>0$ and using this information, 
we directly obtain 
\begin{equation}\label{eq:eigenvalue_2}
\est{u,Du} = \remi{\frac12}\int_{\partial \Omega_{out}} a_n u^2 \diff{s},>0.
\end{equation} 
\remi{and we have an energy estimate.}
We do not discuss \revt{well posedness}, but we
recommend  \cite{nordstrom2017roadmap, nordstrom2019energy} for details regarding that.  
Now, we transfer our analysis to the discrete framework 
and
imitate  this behavior 
discretely. 

\subsubsection*{Discrete Setting}

We have to approximate the spatial operator $D$ and the boundary condition (B.C), i.e. 
$Du+B.C$ by an operator of the form $\mat{M}^{-1}( \mat{Q}-\mat{\Pi}) \vec{u}$
where we apply  SBP operators as defined in Definition \ref{de:SBP}.
The term $\mat{M}^{-1}\mat{Q}\vec{u}$ approximates $Du$ and $\mat{\Pi} \vec{u}$ is used to describe $Bu$ 
weakly. Here, the projection operator $\mat{\Pi}$ works only at the boundary points. 
Note that we must have a $\mat{Q}$ such that $\mat{Q}+\mat{Q}^T$ only contain boundary terms.
Looking at the dissipative nature of $\mat{M}^{-1}\mat{Q}\vec{u}$ amounts to study its spectrum.
The related eigenvalue problem is 
\begin{equation}\label{eq:noDis}
\mat{M}^{-1}(\mat{Q}-\mat{\Pi}) \vec{\tu} =\lambda \vec{\tu}
\end{equation}
We denote by $\vec{\tu}^{*,T}$,  the adjoint of $\tu$ and 
multiply \eqref{eq:noDis} with $\vec{\tu}^{*,T}\mat{M}$  to obtain 
\begin{equation}\label{eq:discrete_setting}
\vec{\tu}^{*,T}(\mat{Q} - \mat{\Pi}) \vec{\tu}= \lambda \vec{\tu}^{*,T}\mat{M}\vec{\tu}= \lambda||\vec{\tu}||_{\mat{M}}^2. 
\end{equation}
We transpose \eqref{eq:discrete_setting}
and add both equations together. This results in 
\begin{equation}\label{eq:discrete_setting_transpose}
\underbrace{\vec{\tu}^{*,T}\left((\mat{Q}+\mat{Q}^T)-(\mat{\Pi}+\mat{\Pi}^T) \right)  \vec{\tu}}_{:=BT}= 2\operatorname{Re}(\lambda)||\vec{\tu}||_{\mat{M}}^2. 
\end{equation}

The boundary terms (BT) correspond to
$\int_{\partial \Omega_{out}} a_n u^2 \diff{s}$ with a properly chosen $\mat{\Pi}$.
Hence, the matrix 
\begin{equation}\label{eq:matrix_eigen}
(\mat{Q}- \mat{\Pi})+ \left( \mat{Q})-\mat{\Pi} \right)^T
\end{equation}
is positive semi-definite, i.e. the eigenvalues 
for the spatial operator  have a strictly positive real parts only.
Note that condition \eqref{eq:matrix_eigen} and \eqref{test}  are the same.
Next, we estimate the boundary operators for a linear system such that the conditions
 in  Theorem \ref{general}
are fulfilled and the pure CG scheme is stable. We start with the continuous energy analysis and
derive the estimate  above. 
Afterwards, we translate the result to the discrete FE framework as done for the scalar one-dimensional case,
but before, we give the following remark:
\begin{remark}[Periodic Boundary Conditions]
As already described in the Introduction \ref{se:Intro}, periodic boundary conditions for hyperbolic problems 
together with a pure Galerkin scheme yields stability issues (von Neumann instability). 
If we consider a Galerkin scheme with the described operators where the used quadrature rule in the numerical integration is the same as the one in the differential operator and periodic boundary conditions in 
a hyperbolic problem, we have the eigenvalues on the imaginary axis, cf. \cite{nordstrom2006conservative}. Therefore, explicit time-integration schemes of order one or two like Euler method or SSPRK22  lead to instability since they 
do not include parts of the imaginary axis in their stability regions. In such a case,  we have to add stabilization terms (diffusion) to the equation.
Further, we want to point out that even with a stable discretization for a linear hyperbolic problem,
 an unbounded error growth is observed if periodic boundary conditions are imposed \cite{nordstrom2007error}. 
\end{remark}

\subsubsection{Systems of Equations} \label{subsec:extension}

Next, we will extend our investigation to the general hyperbolic system 
\begin{equation}\label{eq:2_energy}
\begin{aligned}
 \dpar{U}{t}+A\dpar{U}{x}+B\dpar{U}{y}&=0,\quad&&  (x,y)\in \Omega, t>0\\
 L_\bn U&=G_\bn &&(x,y)\in\partial  \Omega, t>0
\end{aligned}
\end{equation}
where $A,B\in \R^{m\times m}$ are the Jacobian matrices of the system, the matrix $L_\bn\in \R^{q\times m}$ and the 
vector $G_\bn\in \R^q$ are known, $\bn$ is the local outward unit vector, $q$ is the number of boundary conditions to satisfy. 
We assume  that $A,B$ are constant and that the system \eqref{eq:2_energy} is symmetrizable.
There  exists a symmetric and  positive definite  matrix $P$ such that for any vector $\bn=(n_x,n_y)^T$ the matrix 
\begin{equation*}
C_\bn=A_\bn P
\end{equation*}
is symmetric with $A_\bn=An_x+Bn_y$. For arbitrary $n_x$, $n_y$ this implies that 
both $AP$ and $BP$ are symmetric.

%
Using the matrix $P$, one can introduce new variables $V=P^{-1}U$. The original variable can be expressed as
$U=PV$ and the original system \eqref{eq:2_energy} will become 
\begin{equation}\label{eq:3:2}
 P\dpar{V}{t}+AP\dpar{V}{x}+BP \dpar{V}{y}=0 \Longleftrightarrow \dpar{U}{t}+AP\dpar{V}{x}+BP \dpar{V}{y}=0
\end{equation}
The system \eqref{eq:2_energy} admits an energy: if we multiply \eqref{eq:2_energy} by $V^T$, we first get
\begin{equation*}
\begin{split}
\int_\Omega V^T\dpar{U}{t}\; \diff \Omega&=-\int_{\Omega} V^T\bigg ( A\dpar{U}{x}+B\dpar{U}{y} \bigg )\; \diff \Omega=-\int_\Omega V^T \bigg ( AP \dpar{V}{x}+BP \dpar{V}{y}\bigg ) \; \diff \Omega\\
&=-\frac{1}{2}\int_{\partial \Omega} V^T C_\bn V \;\diff  \gamma
\end{split}
\end{equation*}
i.e. setting $E=\frac{1}{2}\int_\Omega V^T U \; \diff \Omega$, we have
$$\dfrac{dE}{dt}+\frac{1}{2}\int_{\partial \Omega} V^T C_\bn V \; \diff  \gamma=0.$$
To understand the role of the boundary conditions, we follow what is usually done for conservation laws, we consider the weak form of 
\eqref{eq:2_energy}: let $\varphi$ be a regular vector function in space and time. We multiply the equation by $\varphi^T$, integrate and get:
\begin{equation*}
\begin{split}
\int_0^T\int_\Omega\varphi^T \dpar{U}{t}\,\diff \Omega\; \diff t-\int_0^T\int_\Omega \bigg ( \dpar{\varphi}{x}^T A+\dpar{\varphi}{y}^T B\bigg ) U\; \diff \Omega\; \diff t+\frac{1}{2}\int_0^T \int_{\partial \Omega} \varphi^T A_\bn U\; \diff \gamma \; \diff t=0.
\end{split}
\end{equation*}
In order to enforce the boundary conditions weakly, we modify this relation by (note that $C_\bn V=A_\bn U$):
\begin{equation}
\label{eq:2}
\begin{split}
\int_0^T\int_\Omega\varphi^T \dpar{U}{t}\,\diff \Omega\; \diff t &-\int_0^T\int_\Omega \bigg ( \dpar{\varphi}{x}^T A+\dpar{\varphi}{y}^T B\bigg ) U\; \diff \Omega\; \diff t\\
&\qquad +\frac{1}{2}\int_0^T \int_{\partial \Omega} \varphi^T A_\bn U\; \diff \gamma \; \diff t\\
&\qquad \qquad =\int_0^T\int_{\partial \Omega} \varphi^T  \Pi_\bn \big (L_\bn(U)-G_\bn\big) \; \diff \gamma \; \diff t.
\end{split}
\end{equation}
The operator $\Pi_\bn$ depends on $\bx=(x,y)\in \partial \Omega$ and $\bn$ the outward unit normal at $\bx\in \partial \Omega$. It is chosen in such a way that:
\begin{enumerate}
\item 
For any $t$, \revt{the image of the boundary operator $L_\bn (U) $ is the same as the image of $\Pi_\bn L_\bn (U)$ in the weak formulation, i.e. there is no loss of boundary information, }
\item If $\varphi=V$, then $\dfrac{\diff E}{ \diff t}< 0$ follows.
\end{enumerate}
\medskip
A solution to this problem is given by the following: 
First, let $C_\bn=X_\bn\Lambda_\bn X_\bn^T$ where $\Lambda_\bn$ is a the diagonal matrix containing the eigenvalues  of $C_\bn$ and $X_\bn$ is the matrix which rows are the right
eigenvectors of $C_\bn$. We have $X_\bn^TX_\bn=\IdxM$ and choose:
\begin{equation}
\label{eq:3}
\Pi_\bn\big ( L_\bn(U)-G_\bn\big )=
 X_\bn \Lambda_\bn^-
 \begin{pmatrix} R_\bn \\ \mathbf{0}_{n-n_0}
 \end{pmatrix} X_\bn^T-
 \begin{pmatrix}G_\bn \\ \mathbf{0}_{n-n_0}\end{pmatrix},
\end{equation}
where $\Lambda_\bn^-$ are the negative eigenvalues only and $n$ denotes the \remi{number of unknowns for the system.} 
Here we have introduced the operator $R_\bn$ which is   $L_\bn$  written using characteristic variables.
\\
\\
\bigskip
In the following, we explain the implementation steps. To a large content we refer to \cite{nordstrom2017roadmap}.
To compute $\Pi$ we first consider again the strong implementation  of the problem. We have
\begin{equation}\label{eq:start}
\frac{\diff{} }{\diff t}V^TU= -\int_{\partial \Omega} V^T C_\bn V\; \diff \gamma.
\end{equation}
Using $C_\bn=X_\bn \Lambda_\bn  X_\bn ^T$, 
we obtain
\begin{equation}
V^T C_\bn V= V^TX_\bn \Lambda_\bn X_\bn^T V= \left(X_\bn ^TV \right)^T \Lambda_\bn  \left(X_\bn^TV \right)
= \begin{pmatrix}
W_\bn^+\\W_\bn^-
\end{pmatrix}^T
 \begin{pmatrix}
\Lambda_\bn^+ & 0 \\0 & \Lambda_\bn^- 
\end{pmatrix}
\begin{pmatrix}
W_\bn^+\\W_\bn^-
\end{pmatrix}
\end{equation}
with $W_\bn^+=\left(X_\bn^TV \right)^+$ are the ingoing waves and 
they have the size of the positive eigenvalues $\Lambda_\bn^+$.
Analogously,   $W_\bn^-=\left(X_\bn^TV \right)^-$ are the outgoing waves with size of $\Lambda_\bn^-$.
A general homogeneous boundary condition is $\revt{W_\bn^-=R_\bn W_\bn^+}$,  since with that, and a proper choice of $R_\bn$,  we get  
\begin{equation}\label{eq:strong_bc_final}
V^T C_\bn V= (W_\bn^+)^T \left( \Lambda_\bn^++R_\bn^T\Lambda_\bn^-R_\bn\right)W_\bn^+\geq 0
\end{equation}
and so the decrease of energy in \eqref{eq:start} if the matrix in the bracket is positive semidefinite. \\

Next, we will impose the boundary conditions weakly. Assume now that we have chosen an $R_\bn$ such that
\begin{equation}\label{eq:Weak_bc}
   \left( \Lambda_\bn^++R_\bn^T\Lambda_\bn^-R_\bn\right) \geq 0.
\end{equation}
\revt{Here, the existence of such $R_\bn$ is ensured through our assumption that our boundary value problem  \eqref{eq:2_energy} is well posed, c.f. \cite{nordstrom2017roadmap}.}
The energy is given  
\begin{equation}\label{eq:energy_general}
\int_{\Omega} V^T \dpar{U}{t}\; \diff \Omega+ \frac{1}{2}\int_{\partial \Omega} V^TA_\bn U \diff \gamma
=\int_{\partial \Omega}V^T \Pi_\bn \left(W_\bn^--R_\bn W_\bn^+ \right)  \; \diff \gamma.
\end{equation}
We add the transpose of \eqref{eq:energy_general} to itself 
and consider 
\begin{equation}
\frac{\diff{} }{\diff t}\int_{\partial \Omega} V^TU\; \diff \Omega 
= -\int_{\partial \Omega}V^TA_\bn U\diff \gamma +\int_{\partial \Omega} V^T \Pi_\bn \left(W_\bn^--R_\bn W_\bn^+ \right)
+  \left(W_\bn^--R_\bn W_\bn^+ \right)^T\Pi_\bn^TV \; \diff \gamma .
\end{equation}
We define $\tilde{\Pi}_\bn$ such that  $V^T \Pi_\bn=(W_\bn^-)^T\tilde{\Pi}_\bn$ and get for the 
integrands
\begin{equation}
-(W_\bn^+)^T\Lambda_\bn^+W_\bn^+-(W_\bn^-)^T\Lambda_\bn^-W_\bn^-+(W_\bn^-)^T\tilde{\Pi}_\bn \left(W_\bn^--R_\bn W_\bn^+ \right)  
+  \left(W_\bn^--R_\bn W_\bn^+ \right) ^T\tilde{\Pi}_\bn^T  W_\bn^-.
\end{equation}
Collecting the terms, we obtain 
\begin{equation}\label{eq:Matrix}
\begin{pmatrix}
W_\bn^+\\W_\bn^-
\end{pmatrix}^T
 \underbrace{\begin{pmatrix}
-\Lambda_\bn^+  &  - R_\bn^T\tilde{\Pi}_\bn^T \\ - \tilde{\Pi}_\bn R_\bn & 
- \Lambda_\bn^- +\tilde{\Pi}_\bn+\tilde{\Pi}_\bn^T
\end{pmatrix}}_{=:WB}
\begin{pmatrix}
W_\bn^+\\W_\bn^-
\end{pmatrix}.
\end{equation} 
We must select $\tilde{\Pi}_\bn$ such that the matrix $WB$ is negative definite.
Now,  let us use the strong condition \eqref{eq:Weak_bc}. 
By adding and subtracting, we obtain 
\begin{equation}\label{eq:Matrix_2}
\underbrace{\begin{pmatrix}
W_\bn^+\\W_\bn^-
\end{pmatrix}^T
 \begin{pmatrix}
R^T_\bn\Lambda_\bn^- R_\bn & - R_\bn^T\tilde{\Pi}_\bn^T  \\ -\tilde{\Pi}_\bn R_\bn
&- \Lambda_\bn^- +\tilde{\Pi}_\bn+\tilde{\Pi}_\bn^T
\end{pmatrix}
\begin{pmatrix}
W_\bn^+\\W_\bn^-
\end{pmatrix}}_{=:Q_w}
\underbrace{-\left((W_\bn^+)^T(\Lambda_\bn^++R_\bn^T\Lambda_\bn^-R_\bn) W_\bn^+\right).}_{\leq 0 \text{ by \eqref{eq:Weak_bc}.}}
\end{equation}
By rearranging and choosing $\tilde{\Pi}_\bn=\Lambda_\bn^-$, we get
\begin{equation*}
Q_w=
\begin{pmatrix}
R_\bn  W_\bn^+\\ W_\bn^-
\end{pmatrix}^T
 \underbrace{ \begin{pmatrix}
1 &  -1 \\ -1&1
\end{pmatrix}}_{=:G_\bn}
\otimes\Lambda^-
\begin{pmatrix}
R_\bn  W_\bn^+\\ W_\bn^-
\end{pmatrix}.
\end{equation*}
Since $G_\bn$ has the eigenvalues $0$ and $2$ and we obtain stability
\begin{equation}\label{eq:energy_end}
\dfrac{d}{dt} \frac{1}{2}
\int_\Omega V^TU\; \diff x +\frac{1}{2}\int_{\partial \Omega} V^TA_\bn U\; \diff \gamma\leq  \frac{1}{2} \int_{\partial \Omega}V^T\Pi_\bn\big ( L_\bn(U)-G_\bn\big )\; \diff\gamma
\end{equation}
thanks to \eqref{eq:Matrix_2} and \eqref{eq:strong_bc_final}.
We will give a concrete example in Section \ref{subsec_R13}.

\section{Numerical Simulations}\label{sec:Numerics}

Here, we demonstrate that a pure Galerkin scheme is stable if 
 we impose the boundary conditions 
weakly and use an adequate boundary operator as described in section \ref{sec:boundary_operator}. 
We consider several different examples and analyze different properties in this context (error behavior, 
eigenvalues, etc.). 
As basis functions, we use Bernstein or Lagrange polynomials of different orders resulting in 
Galerkin schemes of second to fourth order
on triangular meshes. We denote with $B1,\;B2,\;B3$ the Galerkin method using Bernstein polynomials 
with polynomial order $1,2$ or $3$ similar denoting $P1,\;P2,\;P3$ by applying a Lagrange basis. 
The basic implementation is done in the RD framework, see \cite{abgrall2018general}.
The two approaches only differ slightly.
The time integration is done via strong stability preserving Runge-Kutta methods of second to fourth order, see 
\cite{gottlieb2011strong} for details.
We  use always the same order for space and time discretization.

\subsection{Two-Dimensional Scalar Equations}\label{sub:scalar}
We consider a two-dimensional scalar hyperbolic equation of the form
\begin{equation}\label{eq:scalar}
\dpar{U}{t} + \ba(x,y)\cdot \nabla U = 0, \quad (x,y) \in \Omega, \ t > 0, \\
\end{equation}
where $\ba = (a,b)$ is the advection speed and $\Omega$ the domain.
In this subsection, the initial condition is 
 given by
\begin{equation*}
U(x,y,0)=\begin{cases}
          \e{-40r^2},\quad \text{ if }  r=\sqrt{(x-x_0)^2\rev{+}(y-y_0)^2}<0.25, \\
           0, \qquad \text{ otherwise }
          \end{cases}\\
\end{equation*}
It is a small bump with height one located  at $(x_0,y_0)$.
We consider homogeneous boundary conditions $G_\bn\equiv0$
and further let the boundary matrix $\revb{L_\bn}$ be the identity matrix
at the inflow part of $\partial \Omega$ (i.e. $\partial \Omega^-)$. 
 The boundary conditions reads $L_\bn U=U=0$ for $(x,y)\in \partial \Omega^-, \;t >0$
 which means that the incoming waves are set to zero.

\subsubsection*{Linear advection}
In our first test, we are considering the linear advection 
equation in $\Omega=[0,1]^2$.
The advection speed is assumed to be constant. 
The components of the speed vector $\ba$ are given by $(a,b)^T=(1,0)$
and so the flux is given by $\bbf(U)=\bs{a}U$ with $\bs{a}=(1,0)$.
We have inflow / outflow conditions on the left / right boundaries  
and periodic boundary condition on the horizontal boundaries. 
In our first test, we use Bernstein polynomials and a fourth order
C.G. scheme. The boundary operators are computed
using the technique developed in section \ref{sec:boundary_operator}
where the positive eigenvalues are set to zero and the negative ones are used 
in the construction of $\Pi$.
For the time discretization we apply strong stability preserving Runge-Kutta
 (SSPRK) scheme with 5 stages and fourth order 
with CFL=0.3. We use $1048$ triangles.
In figure \eqref{linear_advection_1}, we plot the results at three times. Clearly, the
scheme is stable, also at the outflow boundary. \revs{The maximum value is at the end $1.001$ and the minimum is $-0.0199$ where the starting values are $1.000$ and $0.000$.}
 \begin{figure}[!htp]
 \centering 
   \begin{subfigure}[b]{0.3\textwidth}
    \includegraphics[width=\textwidth]{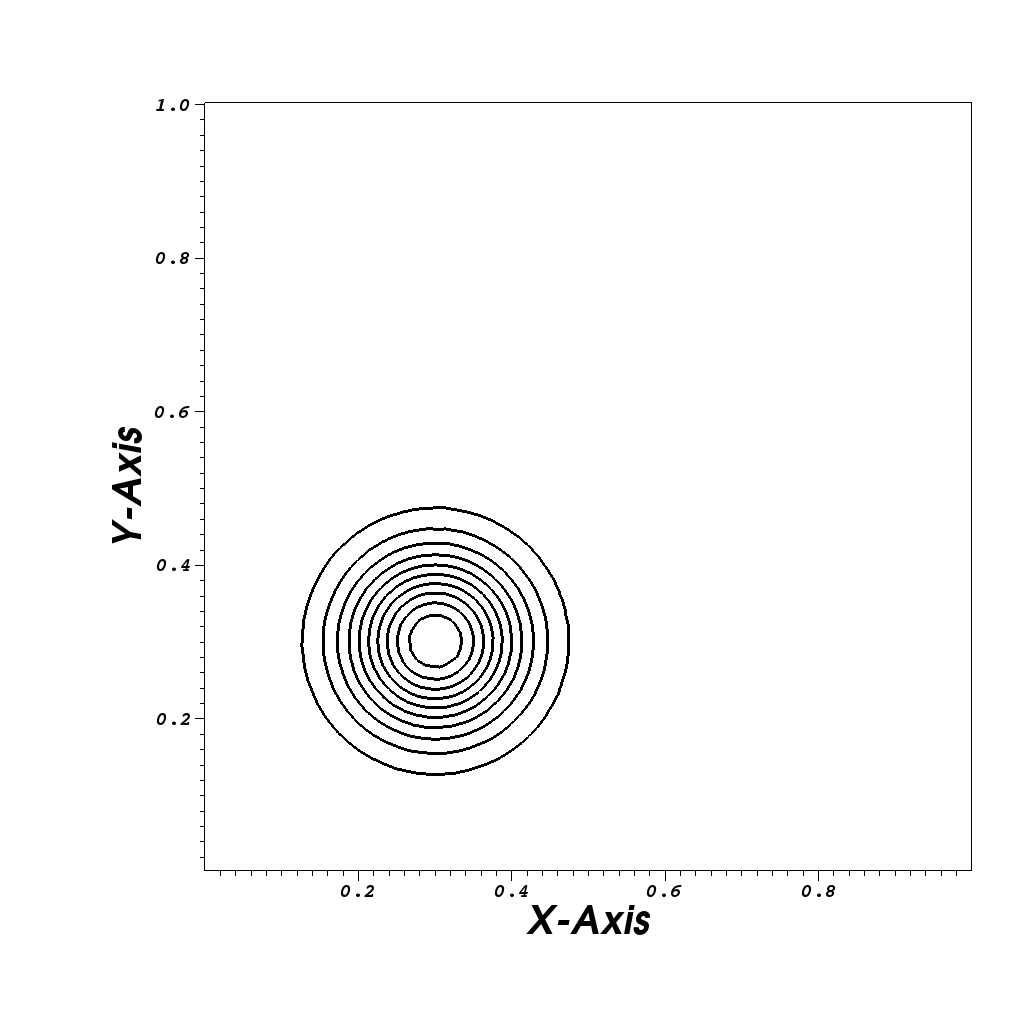}
    \caption{Initial data}
  \end{subfigure}%
   \begin{subfigure}[b]{0.3\textwidth}
    \includegraphics[width=\textwidth]{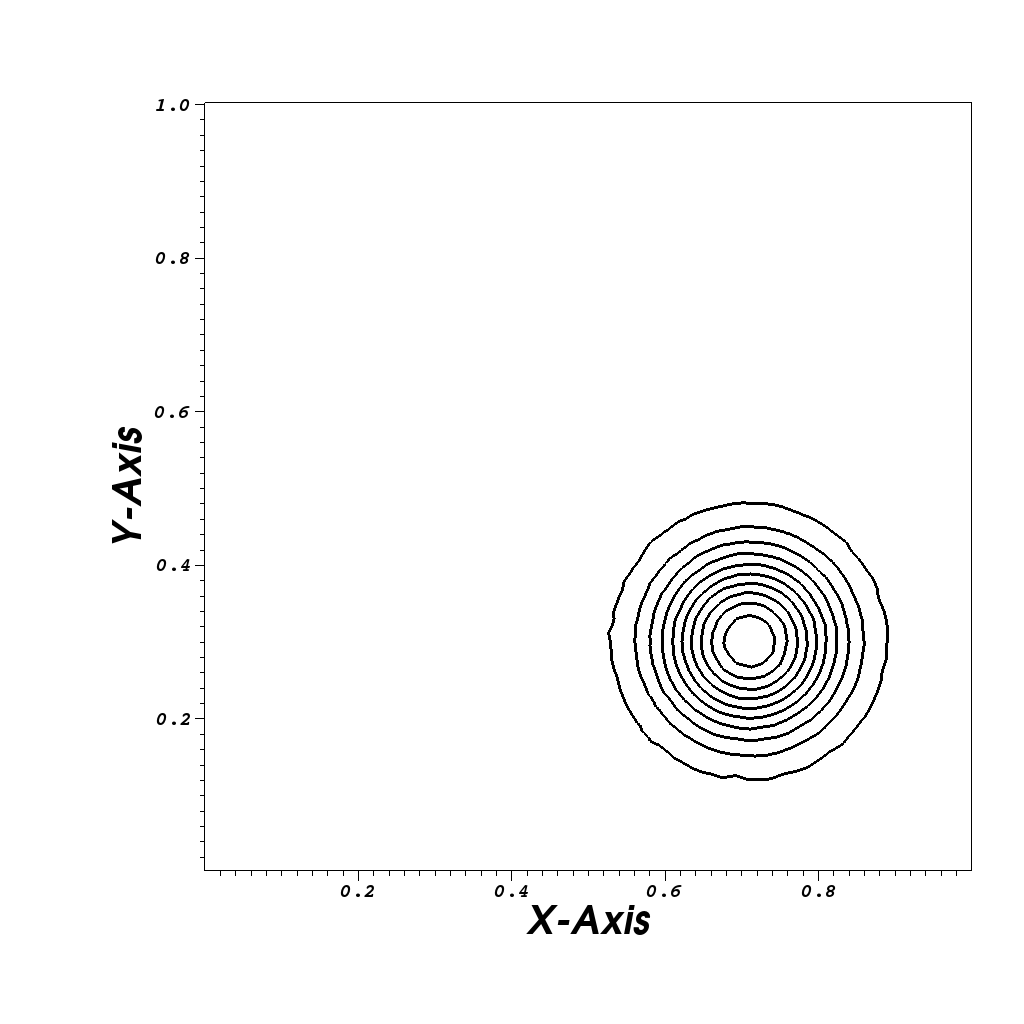}
    \caption{ 100 steps}
  \end{subfigure}%
     \begin{subfigure}[b]{0.3\textwidth}
    \includegraphics[width=\textwidth]{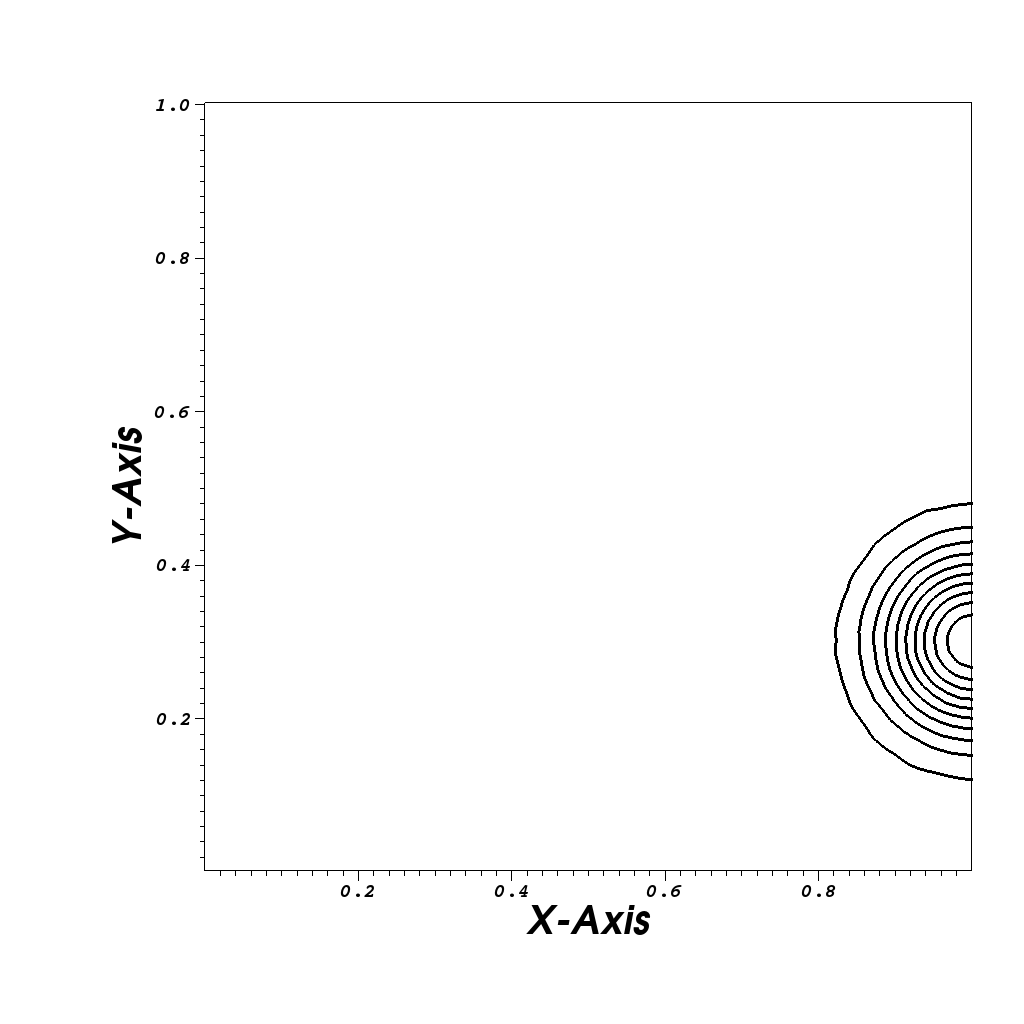}
    \caption{ 173 steps}
  \end{subfigure}%
     \caption{\label{linear_advection_1} 4-th order scheme in space and time}
 \end{figure}
Next, we check the real parts of the eigenvalues of our problem using formula \eqref{test}
for different orders, different bases (Bernstein and Lagrange) and different meshes. 
For the calculation of the eigenvalues of \eqref{test}, 
we use a Petsc routine \cite{petsc-web-page, petsc-user-ref} which can calculate up to 500 eigenvalues\footnote{we have used simple, double and quadruple precision, the results remain the same upto machine precision.}. 
Have in mind that in contrast to  DG and multi-block FD setting, 
the mass matrix in the pure Galerkin scheme is not block diagonal. 
Therefore, we can not split the eigenvalue calculation to each block matrix and have to consider the 
whole matrix $\mat{M}$ and therefore $\mat{Q}$. \rev{Different from before, we consider the complete skew-symmetric spatial operator $\mat{Q}+\mat{Q}^T$ 
in the whole domain and not in one element only, where it is equal to $\mat{B}$.} Every coefficient of the numerical approximation belongs to one degree of freedom 
and we obtain the same number of eigenvalues as number of DoFs are used. 
However, for the coefficients which belong to DoFs inside the domain,  in all calculations
we obtain zero up to machine precision. 
To get  useful results, we decrease the number of elements in the following calculations
and provide only the most negative and positive ones in tables  \ref{P1_eigenvalue}-\ref{B3_eigenvalue}
where we give results with and without the application of the SAT boundary operators.

\begin{table}[!ht]
\centering
  \caption{Eigenvalue of the operators \eqref{eq:matrix_eigen}
  with and without the boundary operators using $P1/B1$ (41DoFs).}
  \label{P1_eigenvalue}
    \small 
  \begin{tabular*}{\linewidth}{@{\extracolsep{\fill}}*4l@{}}
    \toprule
neg. eigen. of \rev{$\mat{Q}+\mat{Q}^T$} &pos. eigen. of \rev{$\mat{Q}+\mat{Q}^T$} & neg. eig.  for BT from \eqref{eq:matrix_eigen}  &pos. eig. for BT from \eqref{eq:matrix_eigen} \\
    \midrule
 $ -0.2317$  &    $ 0.2317 $&$    -0.3135$&$        6.0289 \; \cdot10^{-17}$\\
 $-0.1839$  &     $0.1839  $&$    -0.2555$&$        3.8787\;  \cdot 10^{-17}$\\
$ -0.1250$  &    $ 0.1250  $&$    -0.2317$&$        3.0097\;  \cdot 10^{-17}$\\
$  -6.6074\;  \cdot10^{-2}$&  $ 6.6074\;  \cdot 10^{-2}$& $-0.1848$&$        2.3845\;  \cdot 10^{-17}$\\ 
 $ -5.9935\;  \cdot 10^{-2}$&   $5.9935\;  \cdot 10^{-2}$&$ -0.1839$&$        1.7762\;  \cdot 10^{-17}$\\
 $ -7.2852\;  \cdot 10^{-17}$&   $3.9582\;  \cdot 10^{-17}$&$ -0.1250$&$        1.2997\;  \cdot 10^{-17}$\\
 $ -4.0170\;  \cdot 10^{-17}$&   $3.4527\;  \cdot 10^{-17}$&$ -0.1181$&$        9.7095\;  \cdot 10^{-18}$\\
$  -3.0744\;  \cdot 10^{-17}$& $  2.6742\;  \cdot 10^{-17}$&$  -7.7263\;  \cdot 10^{-2}$&$   9.4396\;  \cdot 10^{-18}$\\
$  -2.9953\;  \cdot 10^{-17}$&$   2.4023\;  \cdot 10^{-17}$&$  -6.6074\;  \cdot 10^{-2}$&$   9.4396\;  \cdot 10^{-18}$\\
$  -2.3732\;  \cdot 10^{-17}$& $  1.8552\;  \cdot 10^{-17}$&$  -5.9935\;  \cdot 10^{-2}$&$   6.6812\;  \cdot 10^{-18}$\\
 $ -1.9299\;  \cdot 10^{-17}$& $  1.2938\;  \cdot 10^{-17}$&$  -7.2894\;  \cdot 10^{-17}$&$   5.6739\;  \cdot 10^{-18}$\\
    \hline
  \end{tabular*}
\end{table}

\begin{table}[!ht]
\centering
  \caption{Eigenvalue of the operators \eqref{eq:matrix_eigen}
  with and without the boundary operators using $B2$ (145DoFs).}
  \label{B2_eigenvalue}
    \small 
  \begin{tabular*}{\linewidth}{@{\extracolsep{\fill}}*4l@{}}
    \toprule
neg. eigen. of $\mat{Q}+\mat{Q}^T$ &pos. eigen. of $\mat{Q}+\mat{Q}^T$& neg. eig.  for BT from \eqref{eq:matrix_eigen}  &pos. eig. for BT from \eqref{eq:matrix_eigen} \\
    \midrule
    $   -0.1343 $&$     0.1343   $&$   -0.1924   $&$     2.8085\;  \cdot 10^{-16} $ \\
      $      -0.1186 $&$      0.1187  $&$    -0.1746  $&$      2.5390\;  \cdot 10^{-16}$ \\
        $     -9.7804\;  \cdot 10^{-2}$&$   9.7804\;  \cdot 10^{-2}$&$ -0.1524  $&$     2.4456\;  \cdot 10^{-16}$ \\
          $  -6.1147\;  \cdot 10^{-2}  $&$ 6.1147\;  \cdot 10^{-2} $&$-0.1343    $&$    2.2495\;  \cdot 10^{-16}$ \\
      $       -5.8452\;  \cdot 10^{-2} $&$  5.8451\;  \cdot 10^{-2} $&$-0.1186  $&$     2.2016\;  \cdot 10^{-16} $\\
        $     -2.6008\;  \cdot 10^{-2} $&$  2.6008\;  \cdot 10^{-2} $&$-0.1088  $&$     2.0043\;  \cdot 10^{-16} $\\
      $       -1.6694\;  \cdot 10^{-2} $&$  1.6694\;  \cdot 10^{-2} $&$ -9.7804\;  \cdot 10^{-2}$&$   1.9144\;  \cdot 10^{-16}$ \\
       $      -1.0862\;  \cdot 10^{-2} $&$  1.0862\;  \cdot 10^{-2} $&$ -6.9258\;  \cdot 10^{-2}$&$   1.8695\;  \cdot 10^{-16}$ \\
        $     -9.4054\;  \cdot 10^{-3}$&$  9.4054\;  \cdot 10^{-3} $&$ -6.1147\;  \cdot 10^{-2}$&$   1.8192\;  \cdot 10^{-16} $\\
         $  -2.7617\;  \cdot 10^{-16}  $&$ 2.8085\;  \cdot 10^{-16} $&$ -5.8452\;  \cdot 10^{-2}  $&$ 1.7825\;  \cdot 10^{-16}$ \\
         $   -2.5757\;  \cdot 10^{-16} $&$  2.5390\;  \cdot 10^{-16} $&$  -3.0333\;  \cdot 10^{-2} $&$  1.7762\;  \cdot 10^{-16}$ \\
    \hline
  \end{tabular*}
\end{table}

\begin{table}[!ht]
\centering
  \caption{Eigenvalue of the operators \eqref{eq:matrix_eigen}
  with and without the boundary operators using $B3$ (313DoFs).}
  \label{B3_eigenvalue}
    \small 
  \begin{tabular*}{\linewidth}{@{\extracolsep{\fill}}*4l@{}}
    \toprule
neg. eigen. of $\mat{Q}+\mat{Q}^T$ &pos. eigen. of $\mat{Q}+\mat{Q}^T$& neg. eig.  for BT from \eqref{eq:matrix_eigen}  &pos. eig. for BT from \eqref{eq:matrix_eigen} \\
    \midrule
$             -9.3746\;  \cdot 10^{-2}  $&$   9.3746\;  \cdot 10^{-2}$&$ -0.1417  $&$      2.2109\;  \cdot 10^{-16} $\\
 $            -8.6774\;  \cdot 10^{-2} $&$  8.6774\;  \cdot 10^{-2} $&$-0.1345    $&$    2.1270\;  \cdot 10^{-16}$\\
  $           -7.7346\;  \cdot 10^{-2} $&$  7.7346\;  \cdot 10^{-2} $&$-0.1260    $&$    2.0375\;  \cdot 10^{-16}$\\
   $          -5.1492\;  \cdot 10^{-2} $&$  5.1492\;  \cdot 10^{-2} $&$ -9.3746\;  \cdot 10^{-2}$&$   1.9449\;  \cdot 10^{-16}$\\
    $         -5.0243\;  \cdot 10^{-2} $&$  5.0243\;  \cdot 10^{-2} $&$ -9.1241\;  \cdot 10^{-2}$&$   1.9060\;  \cdot 10^{-16}$\\
     $        -3.1541\;  \cdot 10^{-2} $&$  3.1541\;  \cdot 10^{-2} $&$ -8.6774\;  \cdot 10^{-2}$&$   1.8710\;  \cdot 10^{-16}$\\
   $          -2.3206\;  \cdot 10^{-2} $&$  2.3206\;  \cdot 10^{-2} $&$ -7.7354\;  \cdot 10^{-2}$&$   1.8543\;  \cdot 10^{-16}$\\
    $         -1.6530\;  \cdot 10^{-2} $&$  1.6530\;  \cdot 10^{-2} $&$ -5.7862\;  \cdot 10^{-2}$&$   1.7657\;  \cdot 10^{-16}$\\
     $        -1.4887\;  \cdot 10^{-2} $&$  1.4887\;  \cdot 10^{-2} $&$ -5.1492\;  \cdot 10^{-2}$&$   1.6555\;  \cdot 10^{-16}$\\
      $      -4.4006\;  \cdot 10^{-3 } $&$ 4.4006\;  \cdot 10^{-3} $&$ -5.0243\;  \cdot 10^{-2} $&$  1.6338\;  \cdot 10^{-16}$\\
       $    -3.0050\;  \cdot 10^{-3} $&$  3.0050\;  \cdot 10^{-3} $&$ -3.6266\;  \cdot 10^{-2}  $&$ 1.6222\;  \cdot 10^{-16}$\\
        $    -2.1187\;  \cdot 10^{-3} $&$  2.1187\;  \cdot 10^{-3} $&$ -3.1541\;  \cdot 10^{-2} $&$  1.6088\;  \cdot 10^{-16}$\\
         $   -1.8526\;  \cdot 10^{-3} $&$  1.8526\;  \cdot 10^{-3} $&$ -2.8790\;  \cdot 10^{-2} $&$  1.5858\;  \cdot 10^{-16}$\\
          $ -2.1399\;  \cdot 10^{-16} $&$  2.2109\;  \cdot 10^{-16} $&$ -2.3210\;  \cdot 10^{-2}  $&$ 1.5731\;  \cdot 10^{-16}$\\
    \hline
  \end{tabular*}
\end{table}

We see from tables \ref{P1_eigenvalue}-\ref{B3_eigenvalue}  that the
the boundary operator decreases the negative 
eigenvalues and forces the positive ones to zero (up to machine precision). 
For third and fourth order, we print only the case using Bernstein polynomials. 
The applications of Lagrange polynomials lead only to slightly bigger amounts
of positive and negative eigenvalues of the $\mat{Q}+\mat{Q}^T$ operator 
(i.e maximum eigenvalue is $ 0.11713334374388217$ for $P3$).
However, the results are similar after applying the SAT procedure, we obtain only negative or zero eigenvalues.\\
We also mention that for higher degrees and more DoFs, we may strengthen the SAT terms to guarantee 
that the eigenvalues are negative and /or forced to zero.
All of our investigations are in accordance with the analysis done in subsection
\ref{section_eigenvalue} and 
all of our calculations demonstrate that a pure Galerkin 
scheme is stable if a proper boundary procedure is used.
\begin{remark}
 Finally, we  did a couple of additional simulations changing both,
  the domain $\Omega$ (circles, pentagons, etc.)
 and the speed vector including also some horizontal movement. 
 All of our calculations remained stable if the boundary approach 
  from section \ref{sec:boundary_operator} was used .
\end{remark}

This is in contradiction of a common belief in the \rev{hyperbolic} research community
about continuous Galerkin schemes.

But what are the reasons for this belief?
In our opinion, one of the major issues is that the chosen quadrature rule in the numerical integration 
differs from the the one used in the differential operators  and without 
artificial stabilization terms the continuous Galerkin scheme collapses,
and the corresponding $\mat{Q}$ matrix does not become almost skew-symmetric.

%
%
%

\subsubsection*{Inexactness of the Quadrature Rule}

To support our statement, we provide the following example.
We consider the same problem as before, but in the Galerkin scheme we 
lower the accuracy of our quadrature rule to \rev{calculate
the mass matrix and the conditions at the boundary procedure. 
Before, we used always a quadrature rule which is accurate up to sixth order.
Then, we lower  the quadrature rule for the surface integral to five, the rest remain the same. Please be
aware that in the Galerkin approach, we apply integration by parts before formulating the variation formulation.}
We decrease the CFL number to $0.01$ for stability reasons.
However, as it is shown in figure \ref{linear_advection_crash}
the scheme crashes after some time even with this super low CFL number. In pictures \ref{fig:crash}, the structure 
of the bump can still be seen, but, simultaneously, 
the  minimum value is $\approx -2.996$ and the maximum value is around $2.7$.

 \begin{figure}[!htp]
 \centering 
   \begin{subfigure}[b]{0.3\textwidth}
    \includegraphics[width=\textwidth]{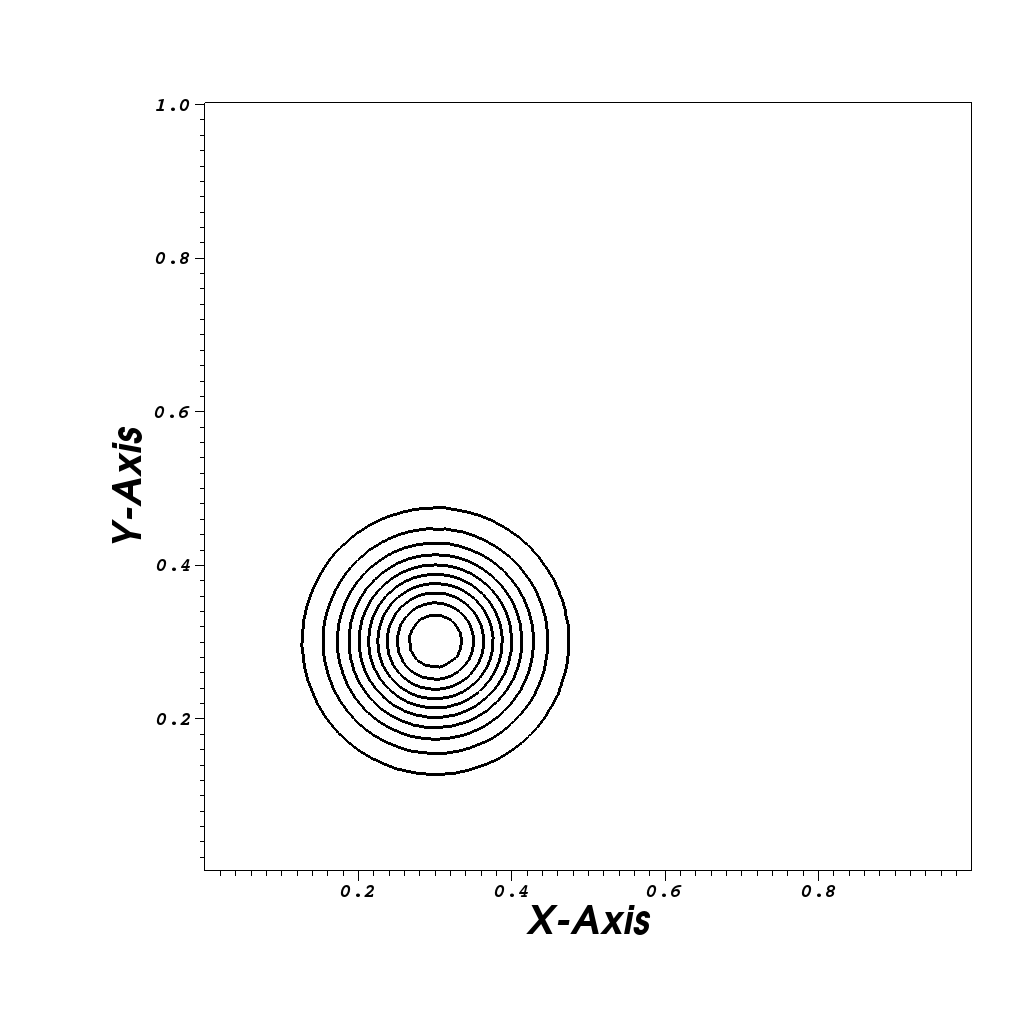}
    \caption{Initial data}
  \end{subfigure}%
   \begin{subfigure}[b]{0.3\textwidth}
    \includegraphics[width=\textwidth]{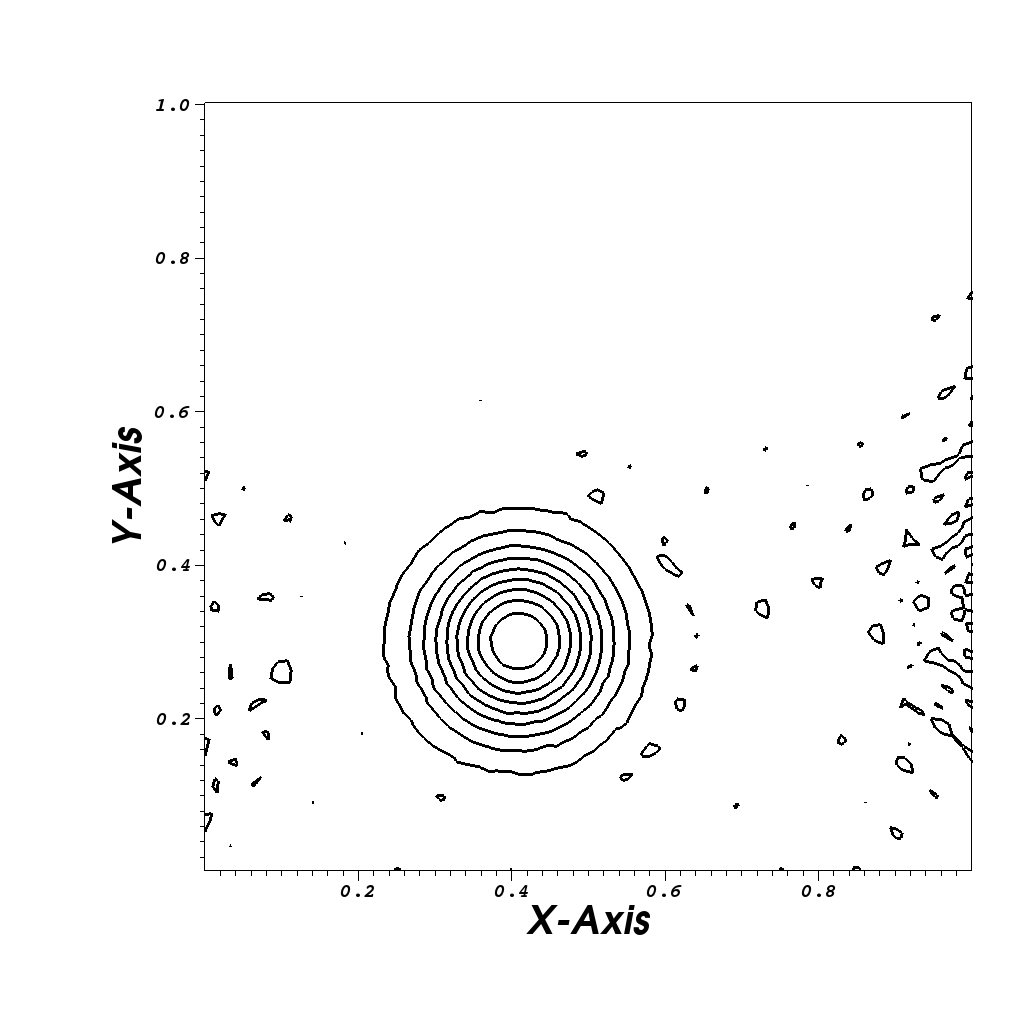}
    \caption{300 steps}
  \end{subfigure}%
     \begin{subfigure}[b]{0.3\textwidth}
    \includegraphics[width=\textwidth]{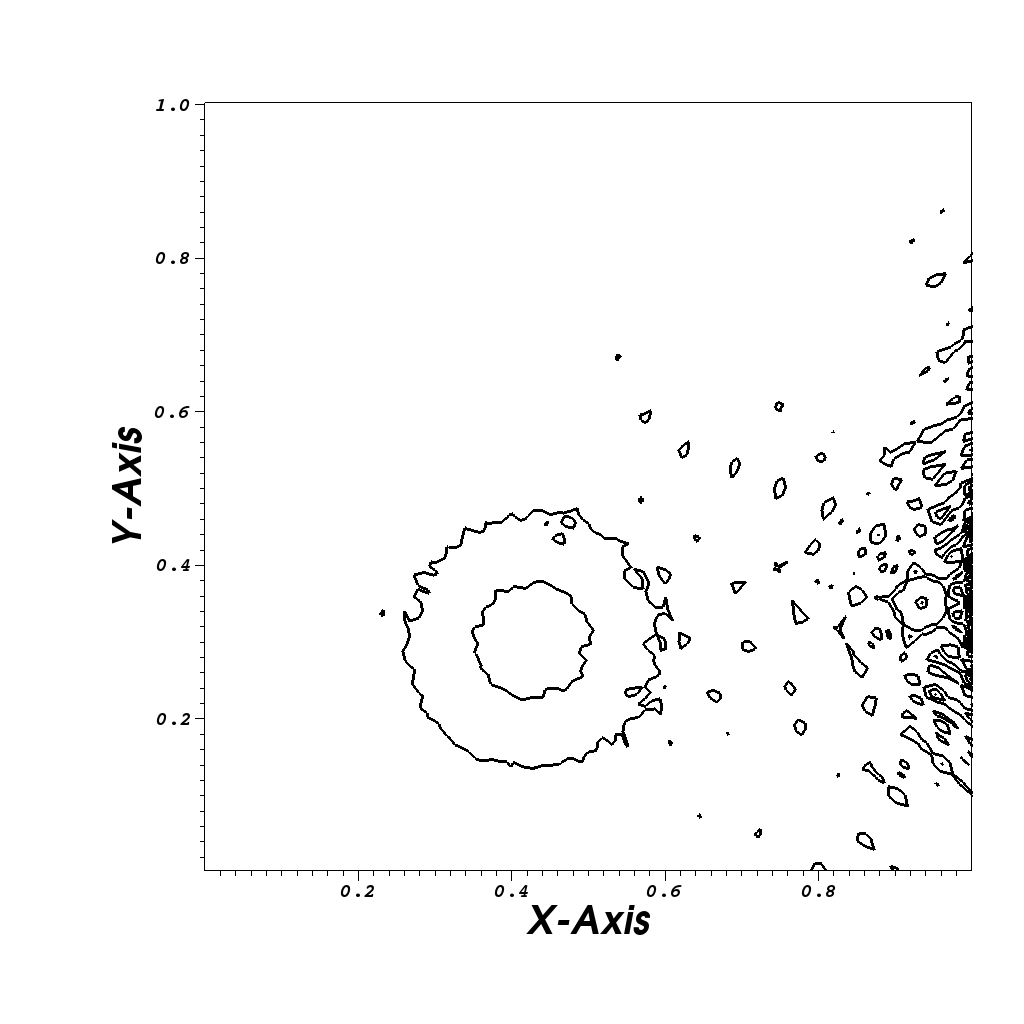}
    \caption{\label{fig:crash} 350 steps}
  \end{subfigure}%
     \caption{\label{linear_advection_crash} 4-th order scheme in space and time}
 \end{figure}
Additional time steps will lead to a complete crash of the test. 
At $400$ steps the maximum value is $74.45$ and the minimum is at $-82.15$.
Here, again nothing has changed from the calculations before except that 
the quadrature rules are changed  which leads to an error in the interior of  the spatial matrix $\mat{Q}$, which cannot be stabilized with the SAT boundary treatment. We will focus on this test again in the second part of the paper series \cite{nonlinear} to demonstrate 
the entropy correction term as presented in \cite{abgrall2018general} and applied in \cite{abgrall2018connection, ranocha2019reinterpretation} can also be seen as a stabilization factor for linear problems.

%
%
%
%
%
%
%
%
%

\subsubsection*{Linear Rotation}
In the next test we consider an advection problem with variable coefficients.
The speed vector has components
\[ a = 2\pi y, \ b = -2\pi x. \]
The initial and boundary conditions are given by 
\begin{align*}
U(x,y,0)&=\begin{cases}
          \e{-40r^2},\quad \text{ if }  r=\sqrt{x^2+(y-0.5)^2}<0.25, \\
           0, \qquad \text{ otherwise }
          \end{cases}\\
 U&=0, \quad (x,y)\in \partial \Omega, \; t>0.
\end{align*}
The problem is defined on the unit disk $\D=\{(x,y)\in \R^2| \sqrt{x^2+y^2}<1 \}$.
The small bump rotates in the clockwise direction in a circle around zero.  
In figure \ref{fig:initial} the initial state is presented where 
figure \ref{fig:mesh_rot} shows the used mesh.
 \begin{figure}[!htp]
 \centering 
   \begin{subfigure}[b]{0.4\textwidth}
    \includegraphics[width=\textwidth]{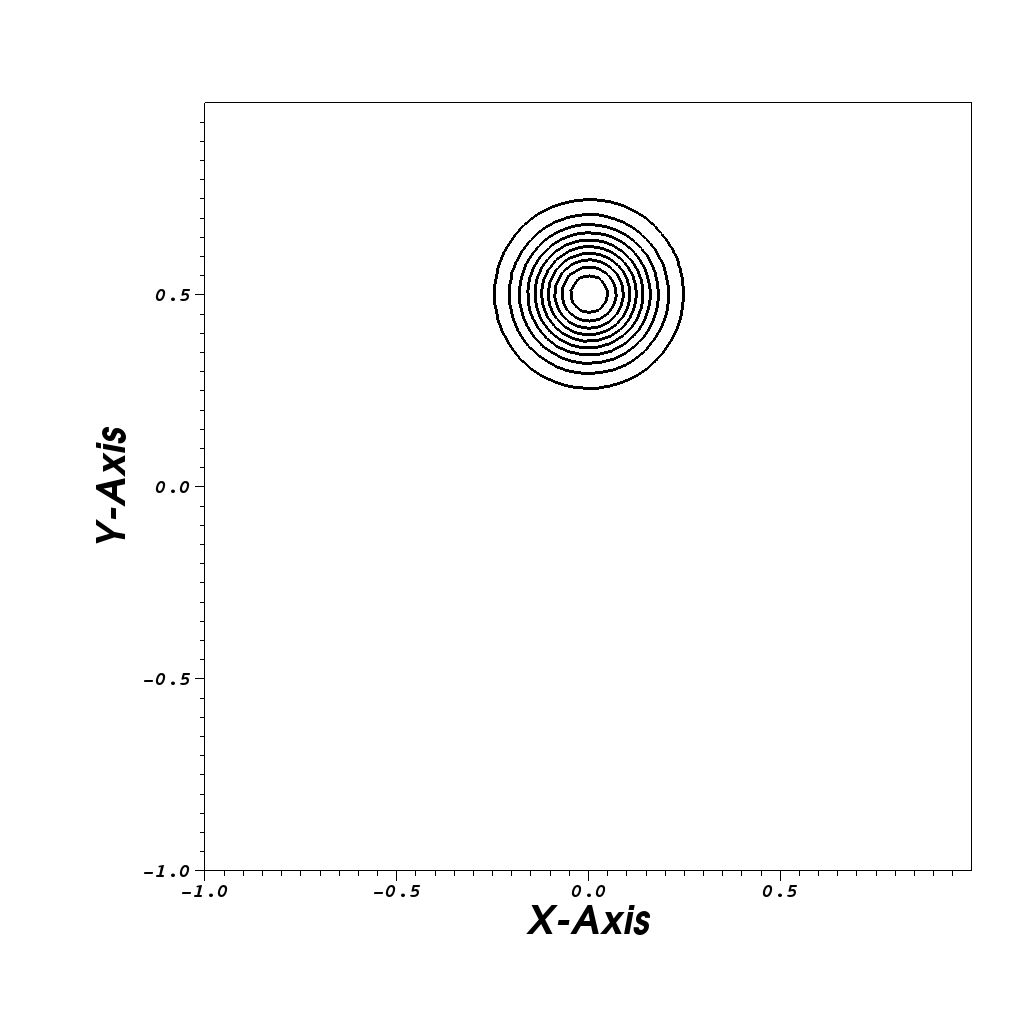}
    \caption{\label{fig:initial} Initial data}
  \end{subfigure}%
   \begin{subfigure}[b]{0.4\textwidth}
    \includegraphics[width=\textwidth]{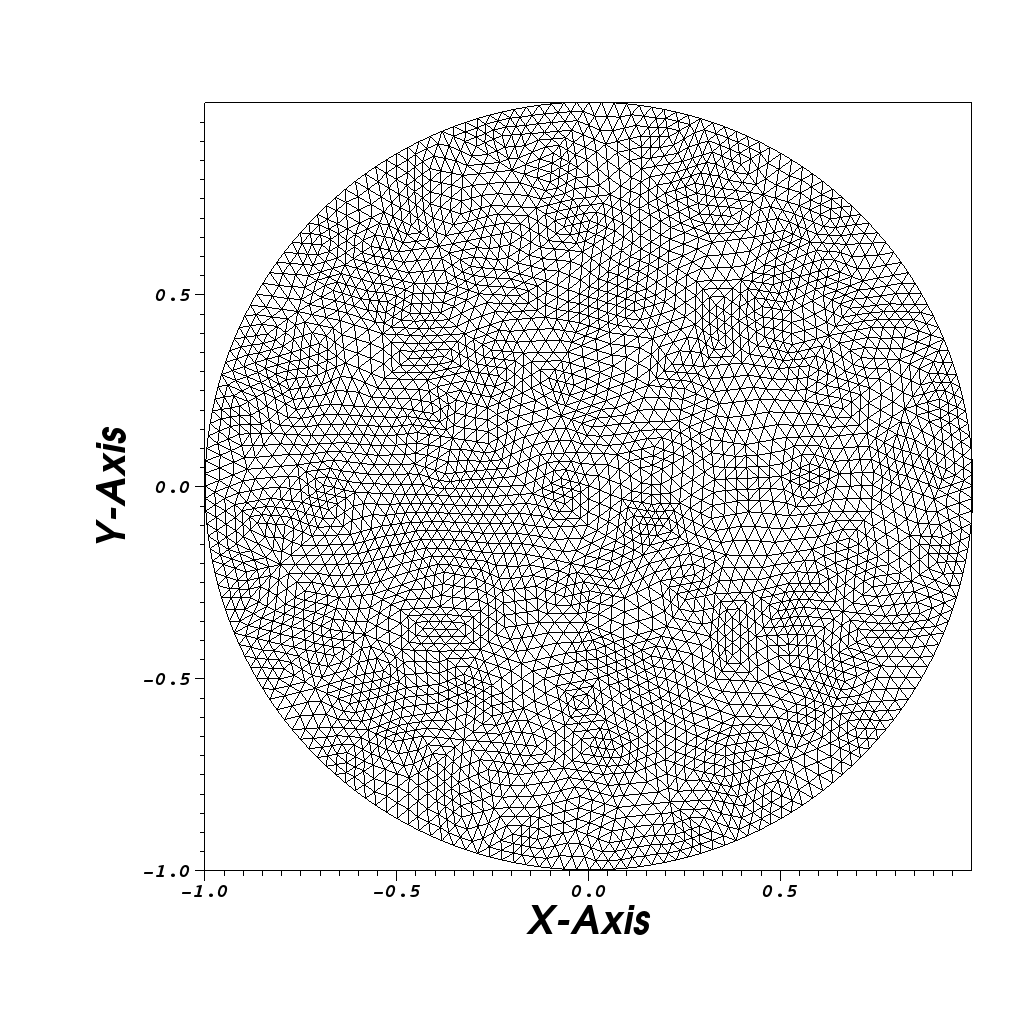}
    \caption{\label{fig:mesh_rot} Mesh}
  \end{subfigure}%
     \caption{4-th order scheme in space and time}
 \end{figure}
We apply an unstructured triangular mesh with $932$  triangles.
In the second calculation $5382$ triangles are used.
The time integration is again done via a SSPRK54 scheme 
with CFL=0.2. 
A pure continuous Galerkin scheme with Bernstein polynomials is used 
for the space discretization. \rev{Due to the variable coefficient problem, we apply the splitting technique as described in \cite{offner2019error,nordstrom2006conservative} and see \ref{re:exactness}. The volume term is split into a symmetric and anti-symmetric part, see for details the mentioned literature. }
The boundary operator is estimated via the approach presented in \ref{subsec:extension}.
In figure \eqref{fig:two}, 
the results are presented after two rotations of the bump. 
Using $932$ triangles, we obtain a  maximum value of $0.993$ and a minimum value 
fo $-0.012$.  Increasing the number of triangles, the maximum value after two rotations is $9.997$ where the 
minimum value is $-0.001$.
This test again verifies that our scheme remains stable only through our boundary procedure.
 \begin{figure}[!htp]
 \centering 
   \begin{subfigure}[b]{0.4\textwidth}
    \includegraphics[width=\textwidth]{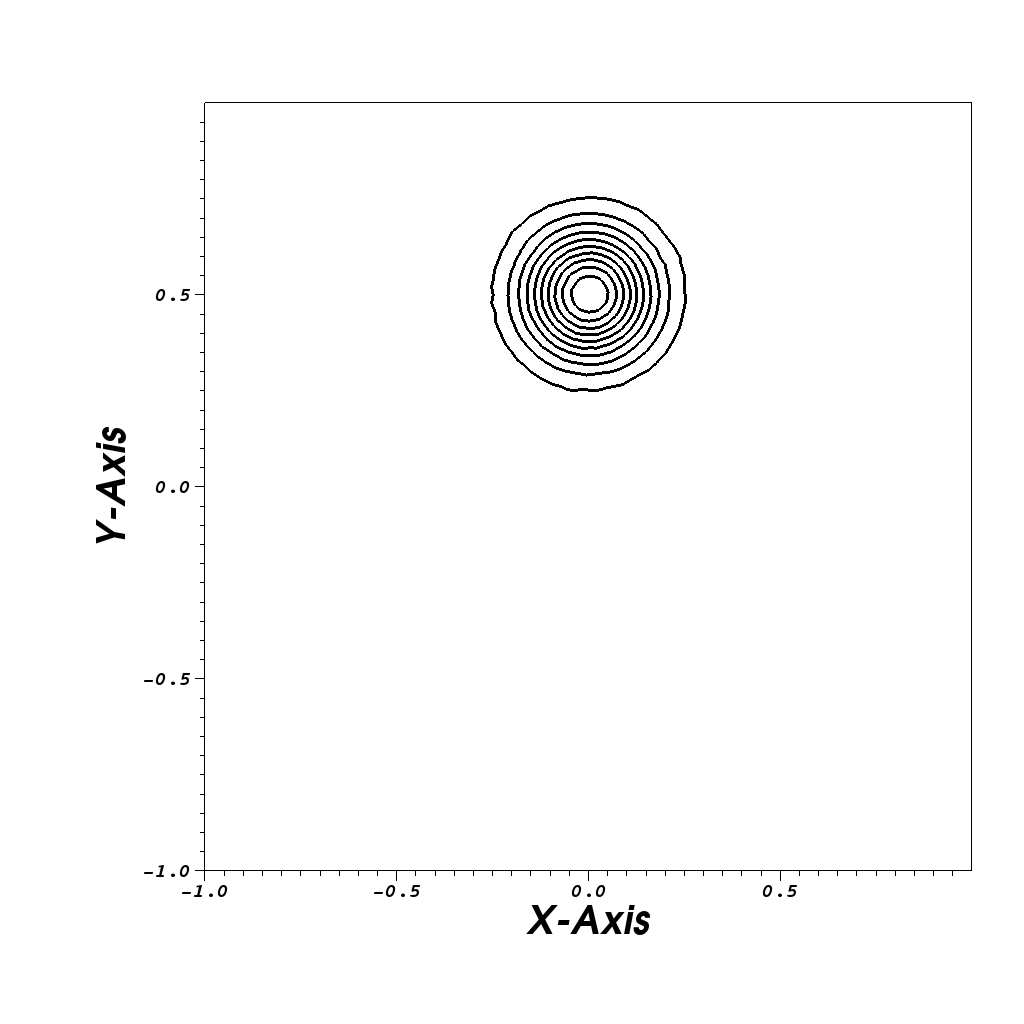}
    \caption{\label{fig:1triangle} $932$  triangles, $\max/\min = 0.993/-0.012$}
  \end{subfigure}%
   \begin{subfigure}[b]{0.4\textwidth}
    \includegraphics[width=\textwidth]{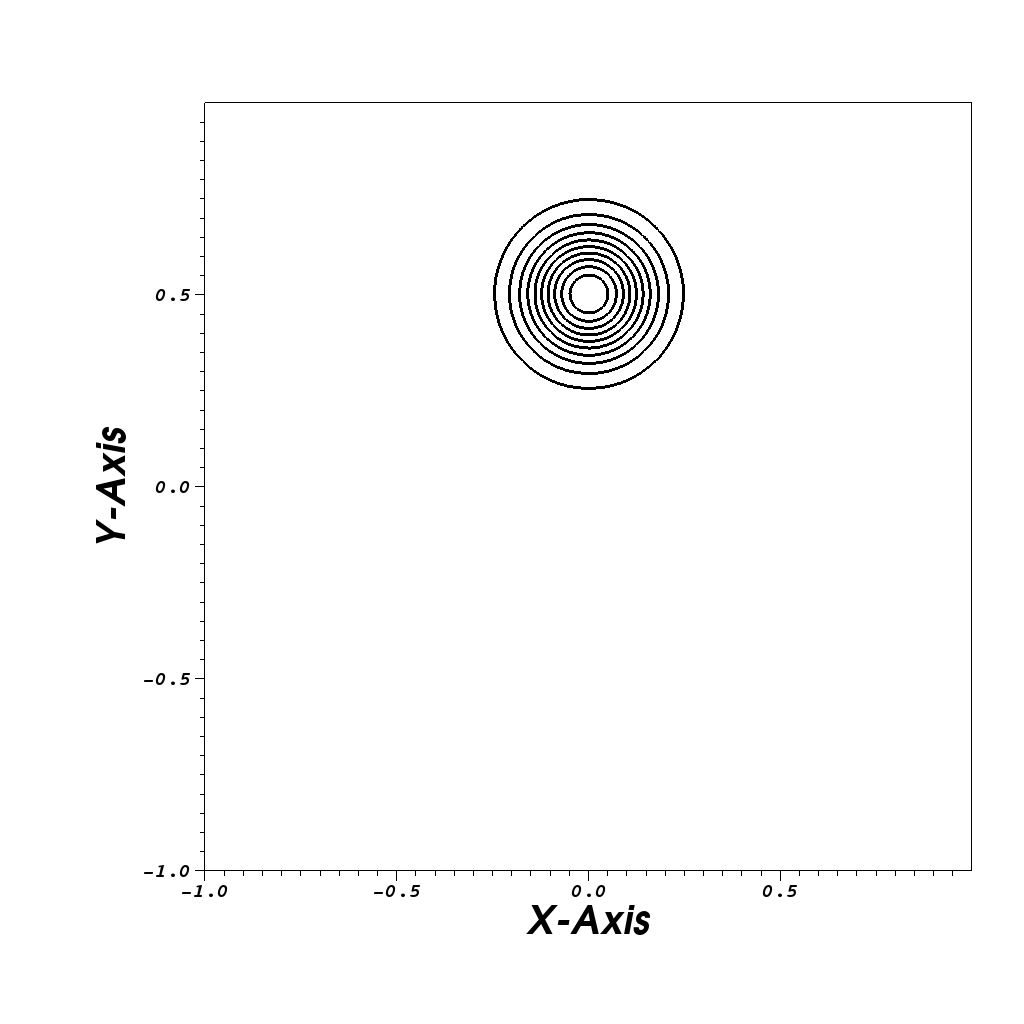}
    \caption{\label{fig:2triangle}  $5382$ triangles,  $\max/\min = 0.997/-0.001$ }
  \end{subfigure}%
     \caption{\label{fig:two} 4-th order scheme in space and time}
 \end{figure}

We compute this problem up to ten rotations for different orders. 
We observe that all of our calculation remain stable both using Lagrange or
Bernstein polynomials as can be seen for example in figure \ref{fig:10}.
In the captions, we mentioned  the respective maximum and minimum values and applied 
$10$ contour lines to to divide the different value regions. Especially, in figure \ref{fig:1:1triangle} 
one can imagine some stability issues. However, this is not the case. Here, the calculations demonstrated some numerical inaccuracies, but the calculation remains stable as can be read of the absolute maximum and minimum values. We recognize also that compared to the others the hight of the bump is decreasing. This 
behavior suggest a certain amount of  artificial dissipation. We obtain the most accurate results using the fourth order
scheme which is not  surprising.

  \begin{figure}[!htp]
 \centering 
     \begin{subfigure}[b]{0.35\textwidth}
    \includegraphics[width=\textwidth]{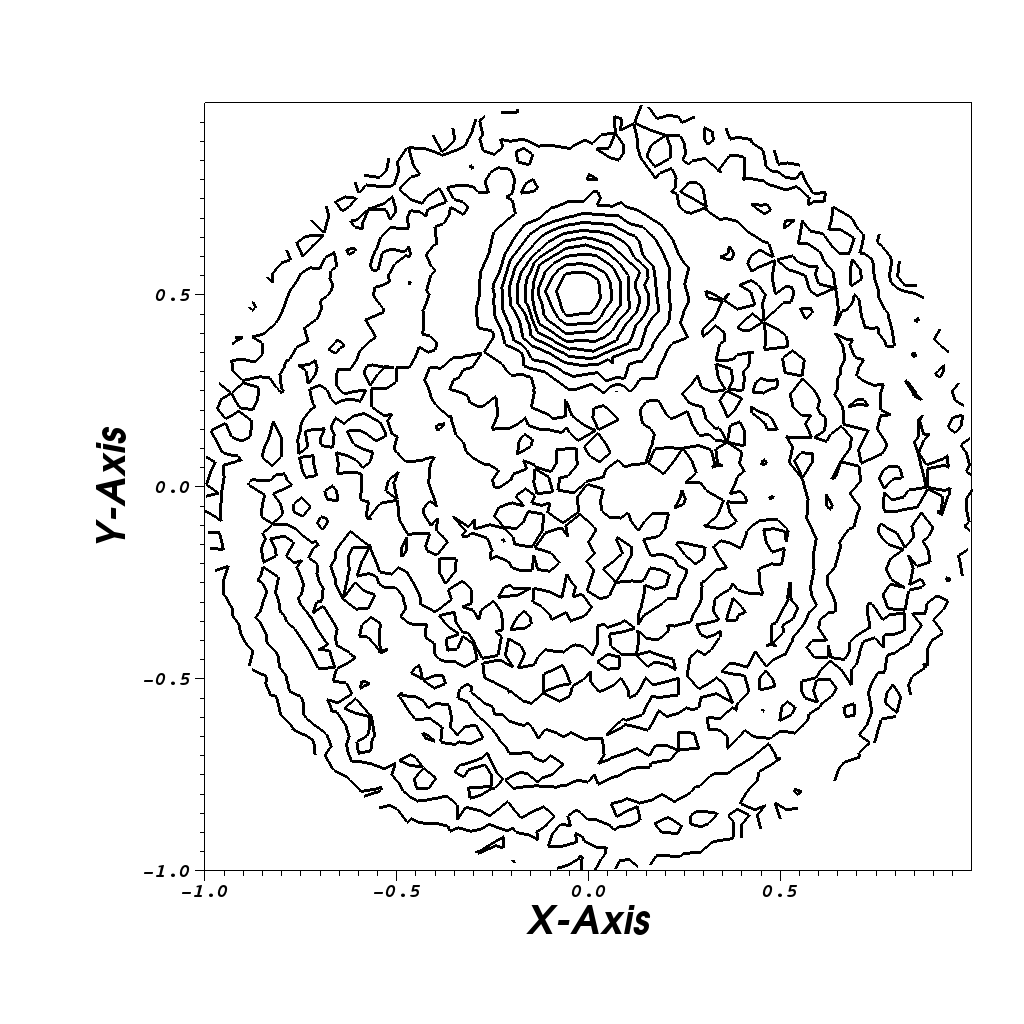}
    \caption{\label{fig:1:1triangle} $5382$  triangles, P1/B1,  \\ $\max/\min = 0.891/-0.057$ }
  \end{subfigure}%
    \begin{subfigure}[b]{0.35\textwidth}
    \includegraphics[width=\textwidth]{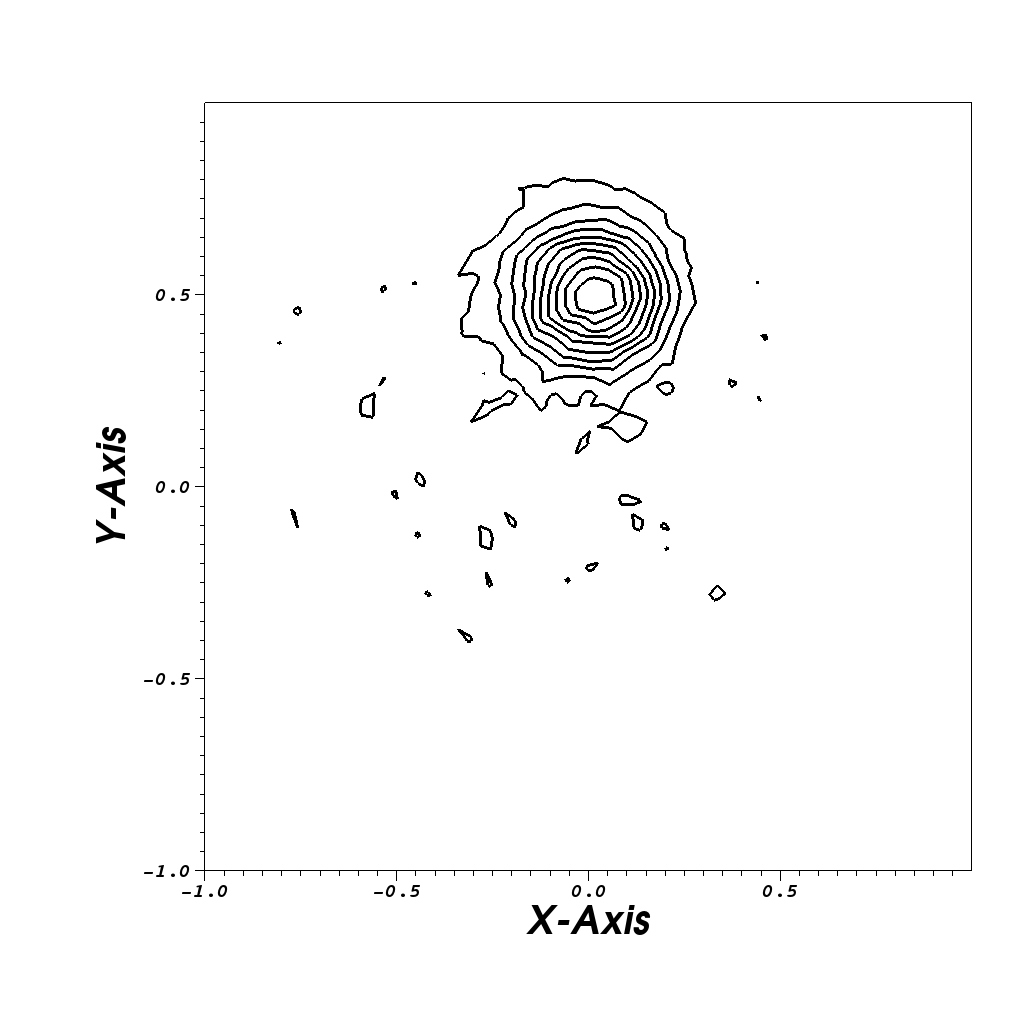}
    \caption{\label{fig:1:2triangle} $932$  triangles, P2,  \\$\max/\min = 0.936/-0.054$ }
  \end{subfigure}%
   \begin{subfigure}[b]{0.35\textwidth}
    \includegraphics[width=\textwidth]{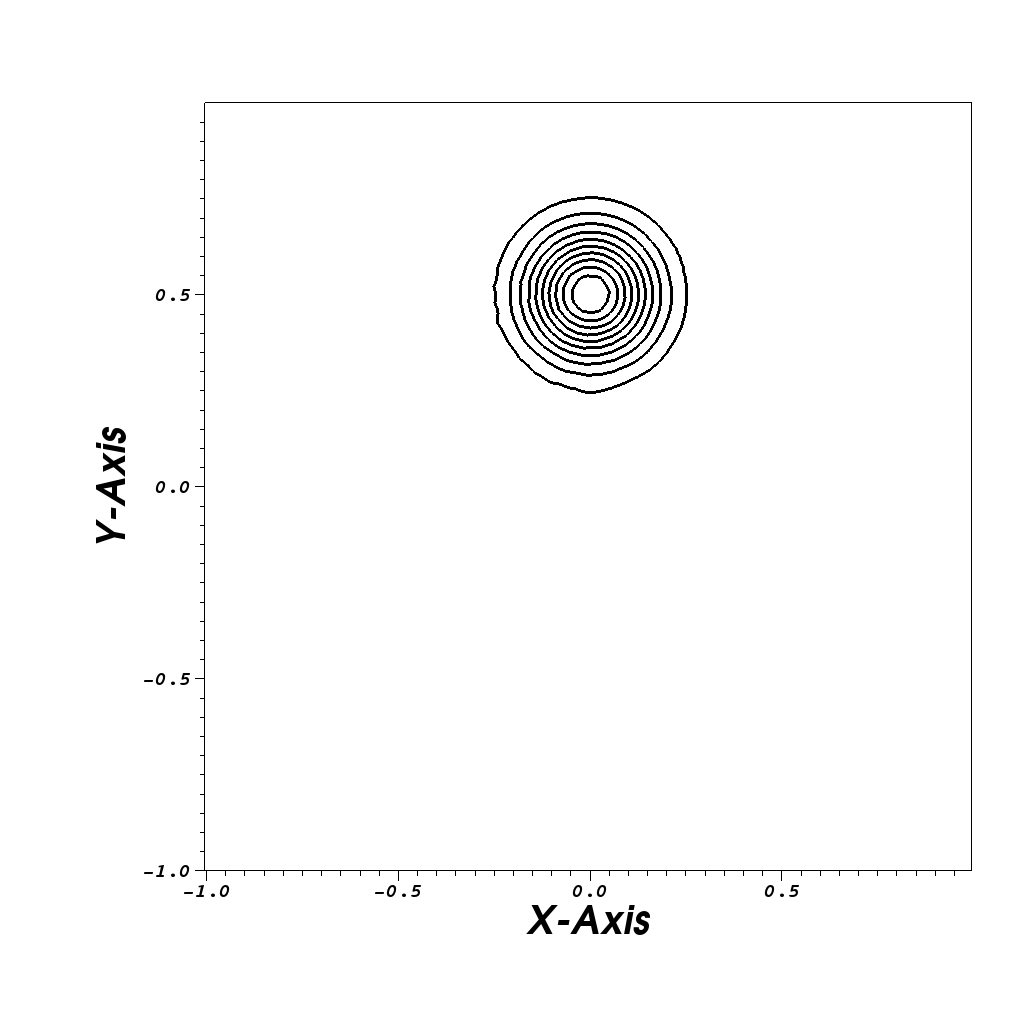}
    \caption{\label{fig:1:3triangle} $932$  triangles, B3,  \\ $\max/\min = 0.994/-0.009$ }
  \end{subfigure}%
     \caption{\label{fig:10} 2,3,4-th order scheme  in space and time}
 \end{figure}
 
Finally,  we analyze the error behavior and calculate the order. 
We use again the SSPRK schemes of the same order. 

\begin{figure}[!htp]
 \centering 
     \includegraphics[width=0.45\textwidth]{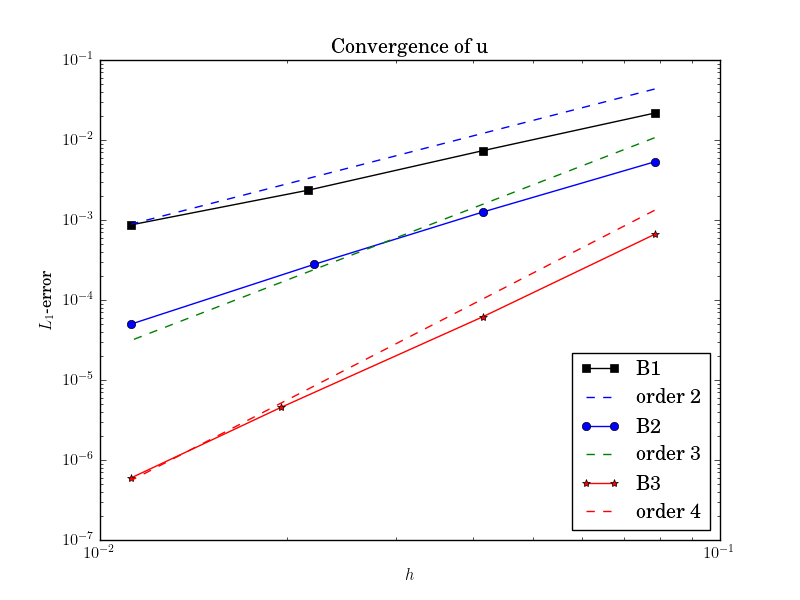}
    \includegraphics[width=0.45\textwidth]{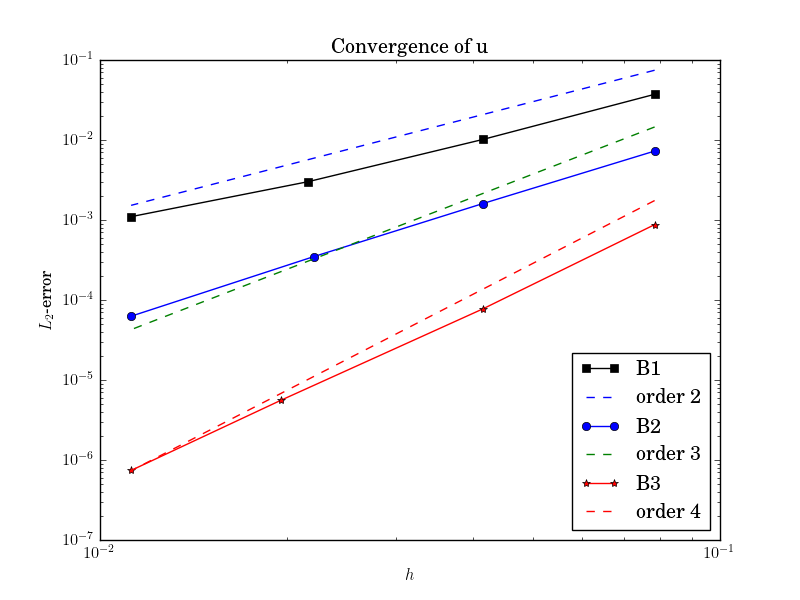}
     \caption{\label{error_short:2} $t=1$, $L_1$-error and $L_2$-error}
 \end{figure}
We recognize a slight decrease of the order similar to the observation made in \cite{zbMATH03835647} which 
was up to our knowledge the first ones who described it.
The investigation of the decreased order of accuracy is not the main focus of this paper, where we 
focus on the stability properties of the pure continuous Galerkin scheme.

\subsection{One-Dimensional Wave Equation}
As a first example for systems with non-homogeneous boundary condition,
we consider the linear wave equation 
\begin{equation*}
 \dfrac{\partial^2 u}{\partial t^2}-\dfrac{\partial^2 u}{\partial t^2} = 0 \quad t>0, x\in (0,1),
\end{equation*}
By applying a change of variables $ \tilde{u}:= \partial_x u$
and $\tilde{v}=-\partial_t u$, the wave equation 
can be rewritten as a first order hyperbolic system of conservation laws
\begin{equation}\label{eq:wave_system}
\begin{aligned}
\dpar{\tilde{u}}{t}+ \dpar{\tilde{v}}{x}&=0,\\
\dpar{ \tilde{v}}{t}+\dpar{\tilde{u}}{x}&=0,
\end{aligned}
\end{equation}
which is sometimes referred to as the one-dimensional acoustic problem. 
 The system \eqref{eq:wave_system} can also be expressed 
through  the linear system 
\begin{equation}\label{eq:wave_system_2}
 \dpar{U}{t}+ A \dpar{U}{x}=0 \text{ with } U=\begin{pmatrix}  \tilde{u}\\
\tilde{v}
\end{pmatrix}
 \text{and coefficient matrix } A= \begin{pmatrix} 
0 & 1\\
1 & 0
\end{pmatrix}.
\end{equation}
which we consider in the following. 
 We assume the  sinusoidal boundary conditions:
 \begin{align*}
 &x=0: & \dfrac{1}{\sqrt{2}}\begin{pmatrix} 1 & 1 \\ 1 &-1 \end{pmatrix}\begin{pmatrix} \tilde{u}\\ \tilde{v} 
 \end{pmatrix}  =\begin{pmatrix} \sin t \\ 0\end{pmatrix} \\
  &x=1: & \dfrac{1}{\sqrt{2}}\begin{pmatrix} 1 & 1 \\ 1 &-1 \end{pmatrix}\begin{pmatrix} \tilde{u}\\ \tilde{v} 
 \end{pmatrix}  =\begin{pmatrix} 0 \\ \sin t \end{pmatrix} \\
 \end{align*}
To determine the boundary operators, we calculate the eigenvalues and the eigenvectors  of 
$A$ following the ideas of Subsection \ref{subsec:extension}. We obtain the eigenvalues $\lambda_{1/2}= \pm1$ and 
$$X=\dfrac{1}{\sqrt{2}}\begin{pmatrix} 1 & 1 \\ 1 &-1 \end{pmatrix}=X^T$$ 
where the rows are the eigenvectors. It is $X^TX=\IdxM$. 
We assume that (the subscripts "0" and "1" refers to the end points of $[0,1]$)
$$
\Pi_0\big ( M_0(U)-g_0\big )=\begin{pmatrix} -R_0 & 1\\0&0\end{pmatrix}X^TU-\begin{pmatrix}\sin t \\ 0\end{pmatrix}
\qquad 
\Pi_1\big ( M_1(U)-g_1\big )=\begin{pmatrix} 0&0\\1&-R_1\end{pmatrix}X^TU-\begin{pmatrix}0 \\ \sin t \end{pmatrix}$$
with \remi{$|R_0|,\;|R_1|<1$}.
Described in \cite{nordstrom2017roadmap}, the problem is well posed in $L^2([0,1])$.
For the time integration, we apply the SSPRK method of third order given in \cite{gottlieb2011strong}
and the space discretization is done via a pure Galerkin scheme of third order
using Lagrange polynomials. The CFL condition is set to 0.4. 
\remi{For $100$ cells and a regular mesh, we have the results displayed in figure
\ref{fig:2}. We tested it up to $t=50$ without any stability problems. The Galerkin scheme gives us  numerical approximations in a way as expected and determined from the theory.}
\begin{figure}[h]
\centering
  \begin{subfigure}[b]{0.4\textwidth}
    \includegraphics[width=\textwidth]{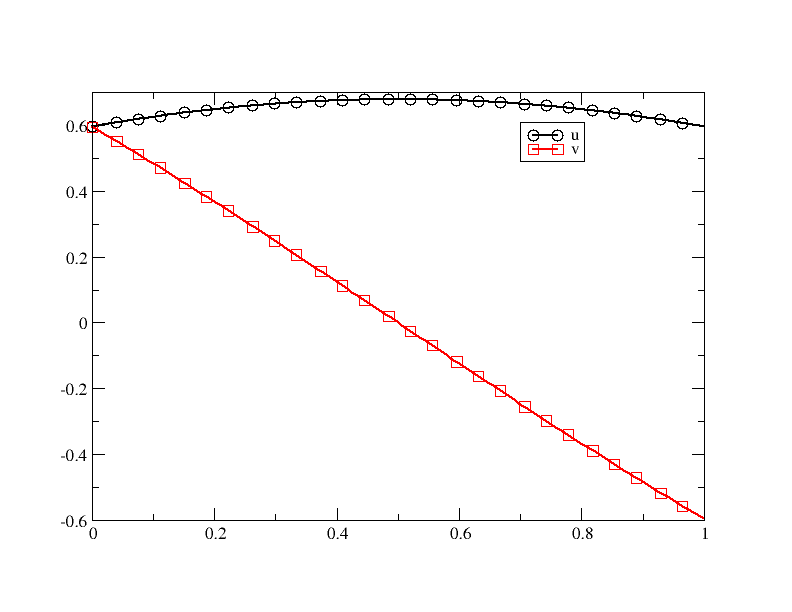}
    \caption{$t=1$}
  \end{subfigure}%
    \begin{subfigure}[b]{0.4\textwidth}
    \includegraphics[width=\textwidth]{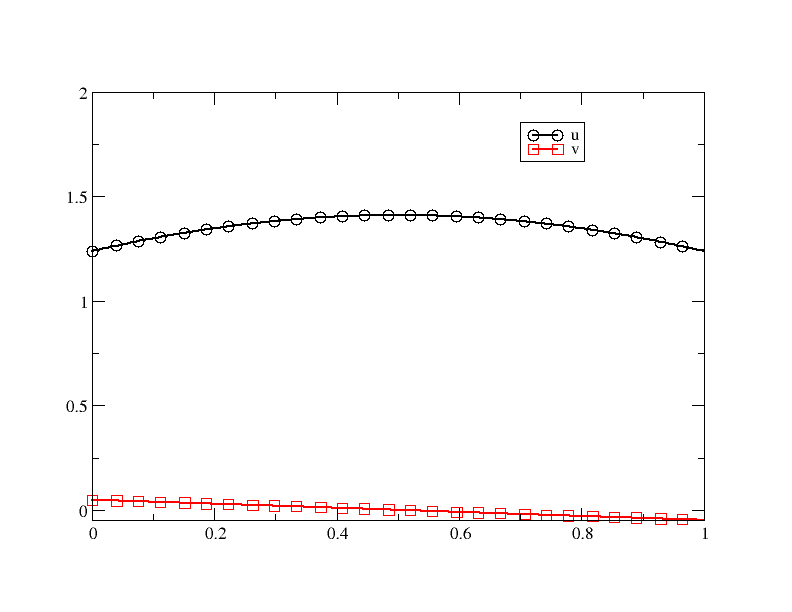}
    \caption{\label{fig:result_1} $t=2$}
  \end{subfigure}\\
    \begin{subfigure}[b]{0.4\textwidth}
    \includegraphics[width=\textwidth]{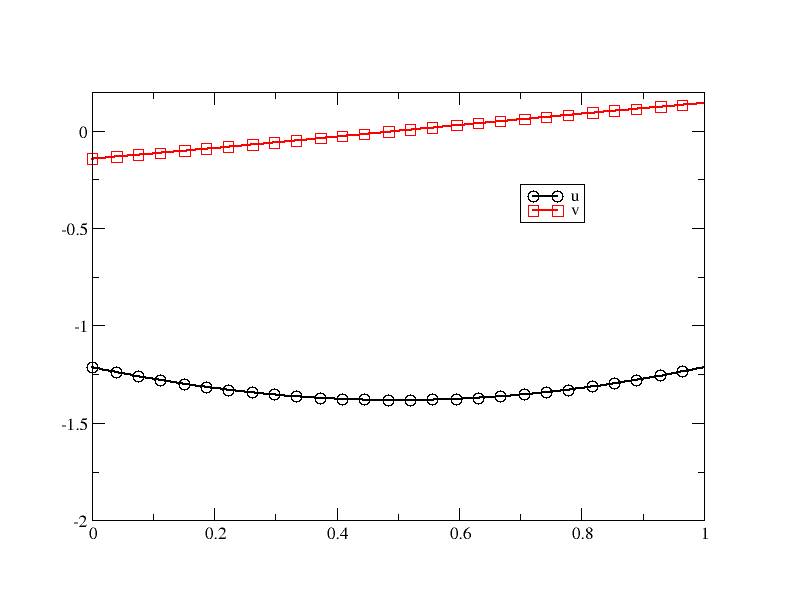}
    \caption{$t=5$}
  \end{subfigure}%
    \begin{subfigure}[b]{0.4\textwidth}
    \includegraphics[width=\textwidth]{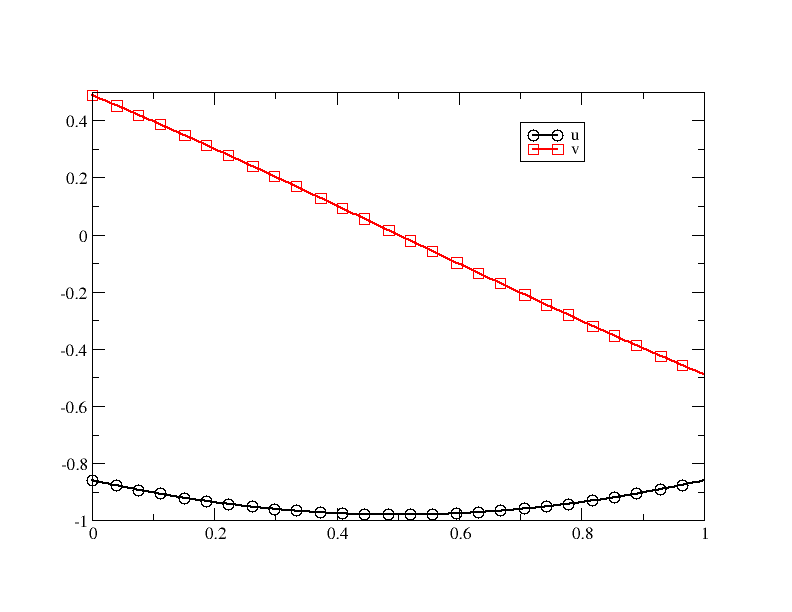}
    \caption{\label{fig:result_22} $t=50$}
  \end{subfigure}
\caption{\label{fig:2} Results for the wave problem \eqref{eq:wave_system} and $t=1,2,5,50$, 3rd order scheme in space and time. We have 100 cells (199 degrees of freedom), CFL=0.1}
\end{figure}
Under the same terms and conditions, we ran the test again now with a random mesh. Figure
\ref{fig:3} demonstrates the results at $t=2$ with a zoom in in figure \ref{fig:zoom}
to highlight the mesh points. Indeed, no visible difference can be seen between figure \ref{fig:result_1}
and \ref{fig:result_2}.
 \begin{figure}[h]
 \centering
\begin{subfigure}[b]{0.4\textwidth}
    \includegraphics[width=\textwidth]{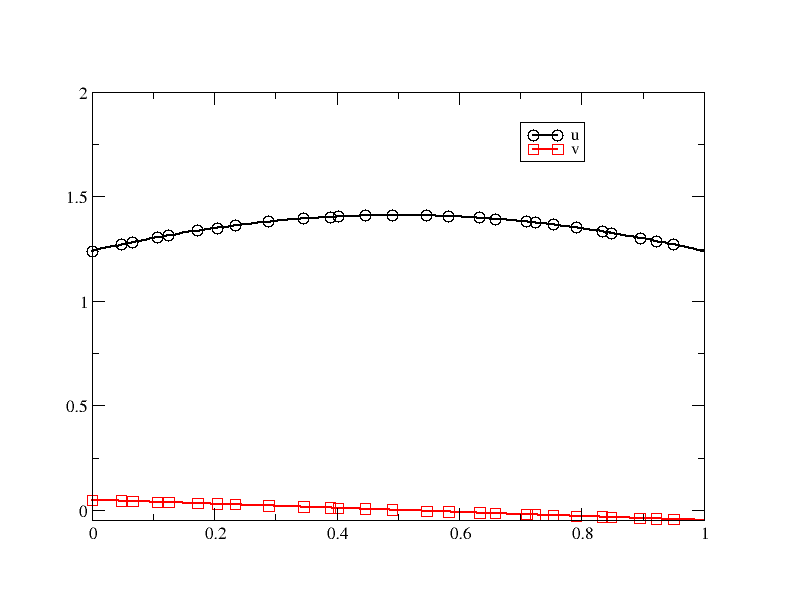}
    \caption{\label{fig:result_2} Whole mesh}
  \end{subfigure}%
    \begin{subfigure}[b]{0.4\textwidth}
    \includegraphics[width=\textwidth]{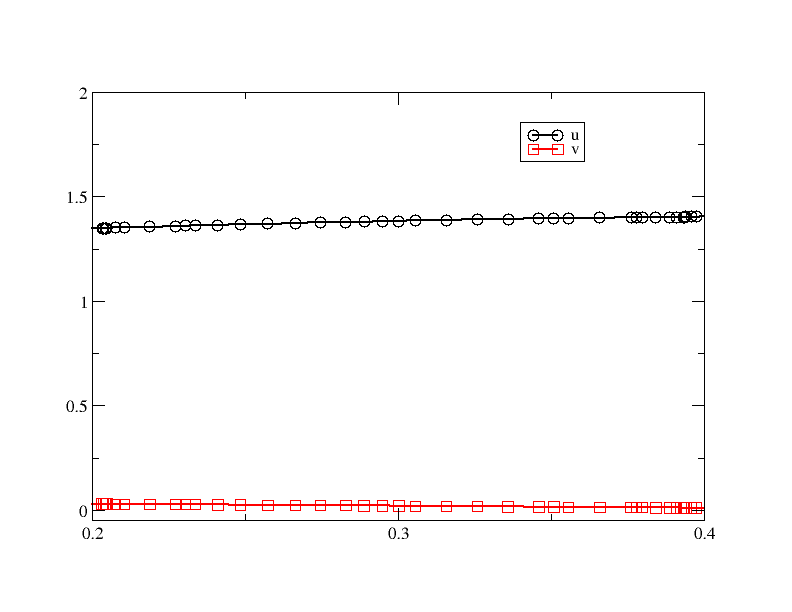}
    \caption{\label{fig:zoom} zoom}
    \end{subfigure}
 \caption{\label{fig:3} Results for the problem and $t=2$, irregular mesh, 3rd order scheme in space and time. We have 100 cells (199 degrees of freedom), CFL=0.1}
 \end{figure}

\subsection{ R13 sub-model for Heat Conduction} \label{subsec_R13}

In our last simulation, we consider the steady R13 sub-model for heat conduction investigated 
 in \cite{torrilhon2016modeling, rana2013robust}. 
It reads 
\begin{equation}
\label{eq:1_R13}
\begin{split}
\text{div } s &=f,\\
\text{ grad }\theta +\text{ div }\mathbf{R}&=-\frac{s}{\tau},\\
\frac{1}{2}\big (\text{ grad} s +(\text{grad }s)^T)&=-\frac{\mathbf{R}}{\tau},
\end{split}
\end{equation}
 in $\Omega=\{(x,y)|\frac{1}{2}\leq \sqrt{x^2+y^2}\leq 1\}\in \R^2$. 
The outer boundary will be denoted by $\Gamma_1$ and the inner circle is $\Gamma_0$.
The process includes a scalar temperature  $\theta\in \R$, a vector values heat flux $s\in \R^2$,
and  a symmetric tensorial variable $\mathbf{R}$ represented by a symmetric $2\times 2$ matrix.
$\tau$ is a constant relaxation time. 

We set
$$s=(s_x,s_y), \mathbf{R}=\begin{pmatrix} R_{xx} & R_{xy}\\R_{xy} & R_{yy}\end{pmatrix}
$$
If $U=(\theta, s, \mathbf{R})$ with $\mathbf{R}=(R_{xx},R_{xy},R_{yy})$
the system \eqref{eq:1_R13} can be rewritten as:
$$A\dpar{U}{x}+B \dpar{U}{y}=0.$$
In the following applications, we will consider the unsteady version of \eqref{eq:1_R13}
$$\dpar{U}{t}+A\dpar{U}{x}+B \dpar{U}{y}=0$$ with boundary conditions that will be detailed in the next part of this section.
The aim is to look for a steady solution of this system, and hence to develop a time marching approach.
With $\alpha \in \R$, the matrix $\cos\alpha A+\sin\alpha B$ reads
\begin{equation}
\label{eq:2_2}
A_\alpha=\begin{pmatrix}
0&\cos\alpha & \sin\alpha& 0 & 0 & 0 &\\
\cos\alpha & 0 & 0 & \cos\alpha & \sin\alpha & 0\\
\sin\alpha & 0 & 0 & 0 & \cos\alpha & \sin\alpha \\
0 & \cos\alpha & 0 &0 & 0 &0\\
0& \frac{\sin\alpha}{2} &\frac{\cos\alpha}{2}& 0& 0& 0\\
0&0&\sin\alpha& 0&0&0
\end{pmatrix}
\end{equation}
and we notice that the system \eqref{eq:1_R13} is symmetrizable.
The symmetrizer is  $P=\diag \left(1,1,1,1,\frac{1}{2},1\right)$
and together, we obtain 
$$A_\alpha P=\begin{pmatrix}
0&\cos\alpha & \sin\alpha& 0 & 0 & 0 &\\
\cos\alpha & 0 & 0 & \cos\alpha & \frac{\sin\alpha}{2} & 0\\
\sin\alpha & 0 & 0 & 0 & \frac{\cos\alpha}{2} & \sin\alpha \\
0 & \cos\alpha & 0 &0 & 0 &0\\
0& \frac{{\sin\alpha}}{2} &\frac{{\cos\alpha}}{2}& 0& 0& 0\\
0&0&\sin\alpha& 0&0&0
\end{pmatrix}=B_\alpha.
$$
$B_\alpha$ is symmetric and to estimate the boundary operator, we need
the eigenvalues and eigenvectors of $A_\bn$. The eigenvectors are
\begin{equation*}
\begin{split}
R=
\begin{pmatrix}
1 & 1 & 0 & 0 & -1 &-\cos^2\alpha \\
\s \cos\alpha & -\s \cos\alpha & -\is \sin\alpha &\is \sin\alpha & 0&0\\
\s \sin\alpha & -\s \sin\alpha & \is \cos\alpha& -\is \cos\alpha & 0&0\\
\cos^2\alpha & \cos^2\alpha & \frac{ \sin(2\alpha)}{2} & -\frac{\sin(2\alpha)}{2}  & 1&\cos(2\alpha)\\
\frac{\sin(2\alpha)}{2} &\frac{\sin(2\alpha)}{2} & -\frac{\cos(2\alpha)}{2} & \frac{\cos(2\alpha)}{2} & 0& \frac{\sin(2\alpha)}{2}\\
\sin^2\alpha& \sin^2\alpha & \frac{\sin(2\alpha)}{2} & \frac{\sin(2\alpha)}{2}& 1 &0
\end{pmatrix}=\begin{pmatrix} R_1, R_2, R_3, R_4, R_5, R_6\end{pmatrix}
\end{split}
\end{equation*}
associated to the eigenvalues
$\lambda=(\s,-\s,\is,-\is,0,0)$.
Through $P$, we can calculate $P^{-1},\;P^{1/2}$ and $P^{-1/2}$ without problems. 
\begin{remark}
 Since the system is symmetrizable, the eigenvectors are orthogonal for the quadratic form associated
 to $P$, i.e. for eigenvectors $r_i\neq r_j$ hold $\langle Pr_i,r_j\rangle=0$,
 where $\langle \cdot, \cdot \rangle$ denotes the scalar product. 
\end{remark}


\subsubsection*{The boundary conditions}
The physical boundary condition follows from Maxwell's kinetic accommodation model.
We have
$$
\begin{pmatrix}
-\alpha \theta+ s_xn_x+s_yn_y-\alpha R_{nn}\\
\beta t_x s_x+\beta t_ys_y+t_xn_xR_{xx}+(t_xn_y+t_yn_x)R_{xy}+t_yn_y R_{yy}
\end{pmatrix}=\rev{L_\bn} U, \qquad U=\begin{pmatrix}
\theta \\ s_x \\s_y \\ R_{xx}\\R_{xy}\\R_{yy} \end{pmatrix}
$$
with normal components $(n_x,n_y)=(\cos\gamma, \sin \gamma)$ and tangential components 
$(t_x,t_y)= (-\sin \gamma,\; \cos \gamma ) $ where 
$\gamma$ is the angle between the $x$-xis and the outward unit normal on $\partial \Omega$.
The accommodation coefficients are given by $\alpha$ and $\beta$.
We have further 
\rev{$R_{\bn\bn}= R_{xx}\cos^2 \gamma+R_{yy} \sin^2 (\gamma) +2R_{xy}\cos( \gamma) \sin(\gamma) $}
and together this gives
$$\rev{L_\bn}=\begin{pmatrix}
-\alpha& \cos \gamma & \sin \gamma & -\alpha \cos^2\gamma & -2 \alpha \cos\gamma \sin \gamma & -\alpha \sin^2\gamma \\
0 & -\beta\sin\gamma & \beta \cos\gamma & -\cos\gamma \sin \gamma & \cos(2 \gamma) & 
\sin \gamma \cos \gamma \end{pmatrix}.
$$
Thanks to this, the boundary conditions on $\Gamma_0$ on $\Gamma_1$ reads 
\begin{equation*}
 \rev{L_\bn} U-\begin{pmatrix} -\alpha \theta_0 \\ -u_x \sin\gamma +u_y\cos\gamma\end{pmatrix}=0 \text{ on } \Gamma_0,\hspace{1cm}
 \rev{L_\bn} U-\begin{pmatrix} -\alpha \theta_1 \\ 0 \end{pmatrix}=0 \text{ on } \Gamma_1,
\end{equation*}
where $\theta_0$ and $\theta_1$ are the prescribed temperatures on the cylinders (boundaries of $\Omega$)
and $u_x, u_y$ denote the prescribed slip velocity.
 To simplify notations, we introduce $G$ as
$$G_\bn(x)=\left \{ \begin{array}{ll}
\begin{pmatrix} -\alpha \theta_0 \\ -u_x \sin\gamma +u_y\cos\gamma\end{pmatrix} & \text{ if } x\in \Gamma_0,\\
\begin{pmatrix} -\alpha \theta_1 \\ 0 \end{pmatrix} & \text{ if } x\in \Gamma_1.
\end{array}\right .
$$
We follow the investigation in subsection \ref{subsec:extension}
and get to the energy balance \eqref{eq:energy_end}
\begin{equation*}
 \int_{\partial \Omega} \bigg ( \frac{1}{2} V^T (An_x+Bn_y) U -
V^T \revs{\Pi} \rev{L_\bn}U \bigg ) \partial \Omega>  - \int_{\partial \Omega} V^T\Pi G_\bn \partial \Omega.
\end{equation*}
In our practical implementation,
we look for $\Pi$ to get energy stability in the homogeneous case.
\remi{Instead of working with the variable transformation $U=PV$, we select  $U=P^{1/2}V$ 
for  convenience  reasons later in the implementation.}
Then 
the condition reads:
\begin{equation}
 \frac{1}{2} V^T (An_x+Bn_y) U -V^T\revs{\Pi} \rev{L_\bn}U= V^T \bigg (   ( \frac{1}{2} A_\bn -\Pi \rev{L_\bn}) \bigg )P^{1/2}V >0.
\end{equation}

One way to achieve this is to assume that $\frac{1}{2} A_\bn-\Pi \rev{L_\bn}$ has the same eigenvectors as
$\frac{1}{2}A_\bn$ and that the eigenvalues are positive, i.e.
$\Pi \rev{L_\bn} -  \frac{1}{2} A_\bn^-$ and $\Pi \rev{L_\bn}$ and $\frac{1}{2} A_\bn^-$
have the same eigenvalues\footnote{Here, we denote again by $-$ the negative eigenvalues.}.  
The idea behind this  is that $(\frac{1}{2}A_\bn-
\Pi \rev{L_\bn})P^{1/2}$ is
positive definite.
 However, this is well-defined under the condition that $\rev{L_\bn}P\rev{L_\bn}^T$ is invertible, and 
 we obtain
\begin{equation}
 \rev{L_\bn}P\rev{L_\bn}^T=\begin{pmatrix}
 1+2\alpha^2 & 0\\
 0& \frac{1}{2}+\beta^2
 \end{pmatrix}.
\end{equation}
This matrix is always invertible since its determinant  is always positive.  A solution to the problem is 
$\Pi \rev{L_\bn} P \rev{L_\bn}^T=R D L P^{1/2}\rev{L_\bn}^T$ with $D\leq \frac{1}{2}\Lambda^-$
so $\Pi= R DL P^{1/2} \rev{L_\bn}^T (\rev{L_\bn}P\rev{L_\bn}^T)^{-1}$ with $D\leq \frac{1}{2}\Lambda^-$
and using the transformation with $P^{1/2}$, we obtain:
 $$   \left(\frac{1}{2} A_\bn-\revs{\Pi} \rev{L_\bn}\right)P^{1/2}V=\lambda P^{1/2}V,$$
 i.e.
  $$\left( \frac{1}{2} A_\bn-\lambda I\right)P^{1/2}V=\Pi \rev{L_\bn}P^{1/2}V$$
  that is 
  \begin{equation}\label{eq:boundary_operator_R13}
  \left( \frac{1}{2} A_\bn-\lambda I\right)P\rev{L_\bn}^T(\rev{L_\bn}P\rev{L_\bn}^T)^{-1}V=\Pi V
 \end{equation}
Using $U$ instead of $V$ in the implementation, we have to multiply $\Pi$ with $P^{-1/2}$. 
\begin{remark}\label{5.3}
 Another way to determine  $\Pi$, we choose $\delta <0$ such that 
  $\Pi \rev{L_\bn}P^{1/2}-\frac{1}{2} A_\bn P^{1/2} =\delta \text{Id}$, and thus yields 
  \begin{equation*}
   \Pi=\big ( \delta P^{-1/2} +\frac{1}{2} A_\bn\big ) \rev{L_\bn}^T (\rev{L_\bn} \rev{L_\bn}^T)^{-1}.
  \end{equation*}
   However, this is well-defined under the condition that $\rev{L_\bn}\rev{L_\bn}^T$ is invertible.
  We obtain
\begin{equation}
 \rev{L_\bn}\rev{L_\bn}^T=\begin{pmatrix} \frac{1}{4} (4+9\alpha^2-\alpha^2\cos 4\gamma)& -\frac{1}{4}\alpha\sin 4\gamma \\
-\frac{1}{4}\alpha\sin 4\gamma & \frac{1}{4} (3+4\beta^2+\cos 4\gamma) \end{pmatrix}.
\end{equation}
The matrix is always invertible since elementary calculations yield to an estimation of the 
determinate which can be shown to be bigger than  0.5.
\end{remark}

\subsubsection*{Concrete example}
We have explained how we estimate the boundary operator
in the above equation \eqref{eq:boundary_operator_R13}.
In the test, we have set the accommodation coefficients $\alpha=3.0$ and $\beta = -0.5$.
The temperature at the boundaries are given by $\theta_0=0$ and $\theta_1=1$. 
Further, we have $u_x=1$
and $u_y=0$. The relaxation time is set to $0.15$.
Again, we use a continuous Galerkin scheme together with the above developed boundary procedure. 
The term $\delta$ is set to $-2$  and  the CFL number is $0.1$.
We  ran the problem up to steady state with a RK scheme.
In figure \ref{fig:R13} we show the mesh and also the result at steady state
using a coarse grid. The number of triangles is $400$. 
The problem is elliptic and  smooth which cannot be seen in this 
first picture since the mesh is too coarse.
 \begin{figure}[h]
 \centering
    \includegraphics[width=0.4\textwidth]{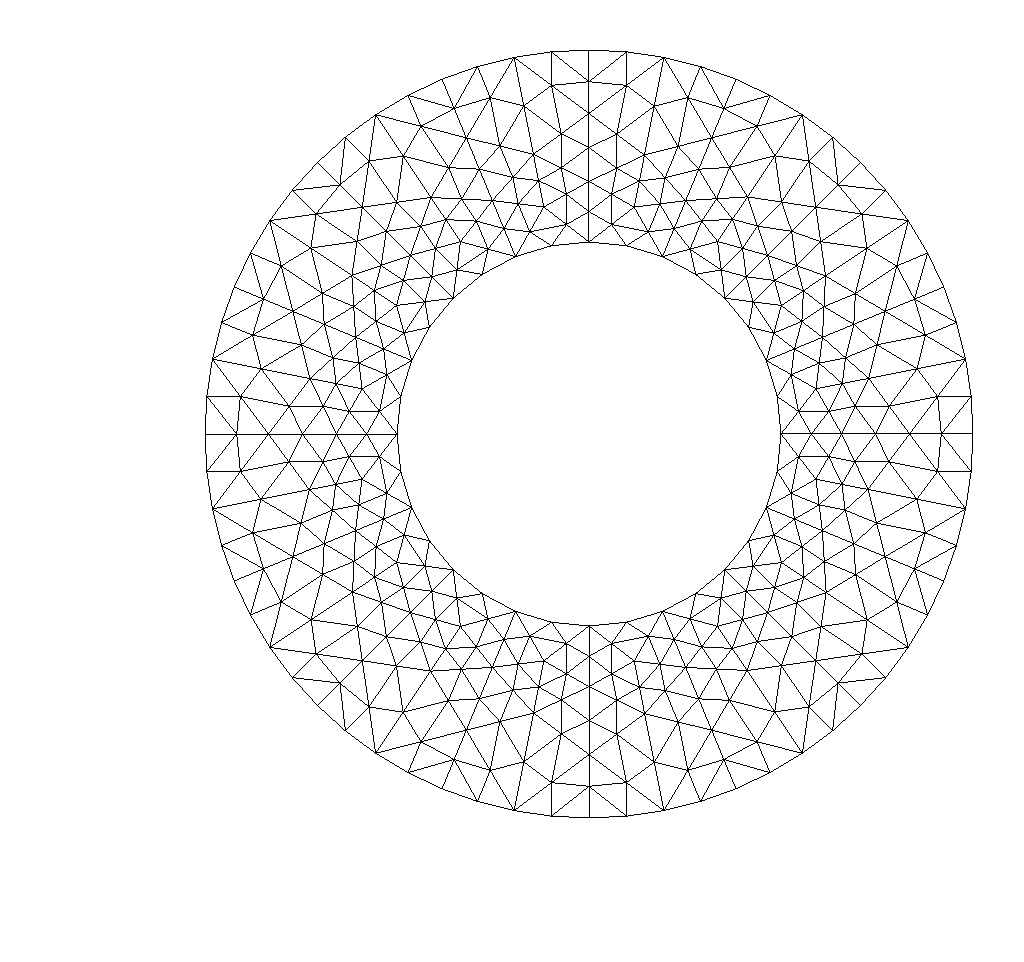}
    \includegraphics[width=0.4\textwidth]{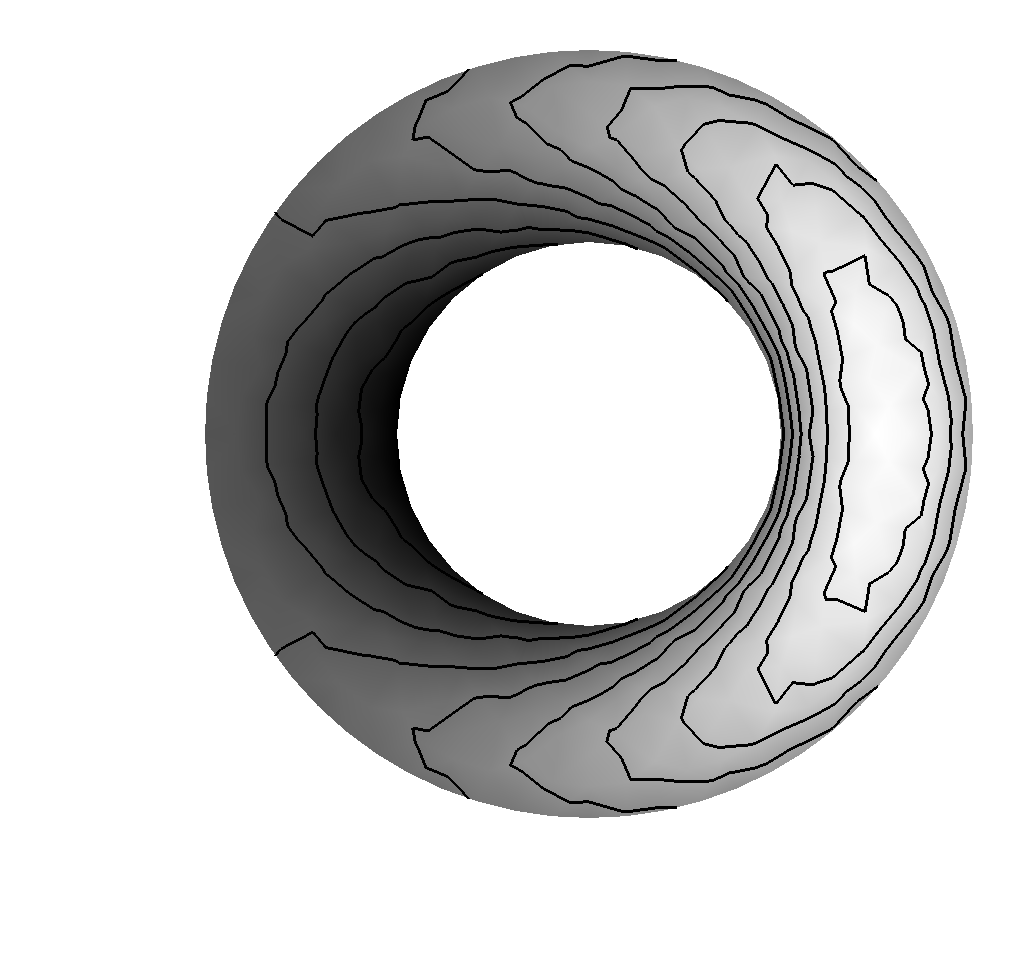}
 \caption{\label{fig:R13} Mesh and  steady state ($t=10)$, 3rd order scheme}
 \end{figure}

In the second test, we increase the number of elements in the mesh. Now,
we are using 5824 elements and also Bernstein polynomials of second order.
The mesh and the result are presented in figure \ref{fig:R13_2}.
Here, we recognize the smooth behavior
and the scheme remains stable only through the above described boundary procedure.

 \begin{figure}[h]
 \centering
    \includegraphics[width=0.4\textwidth]{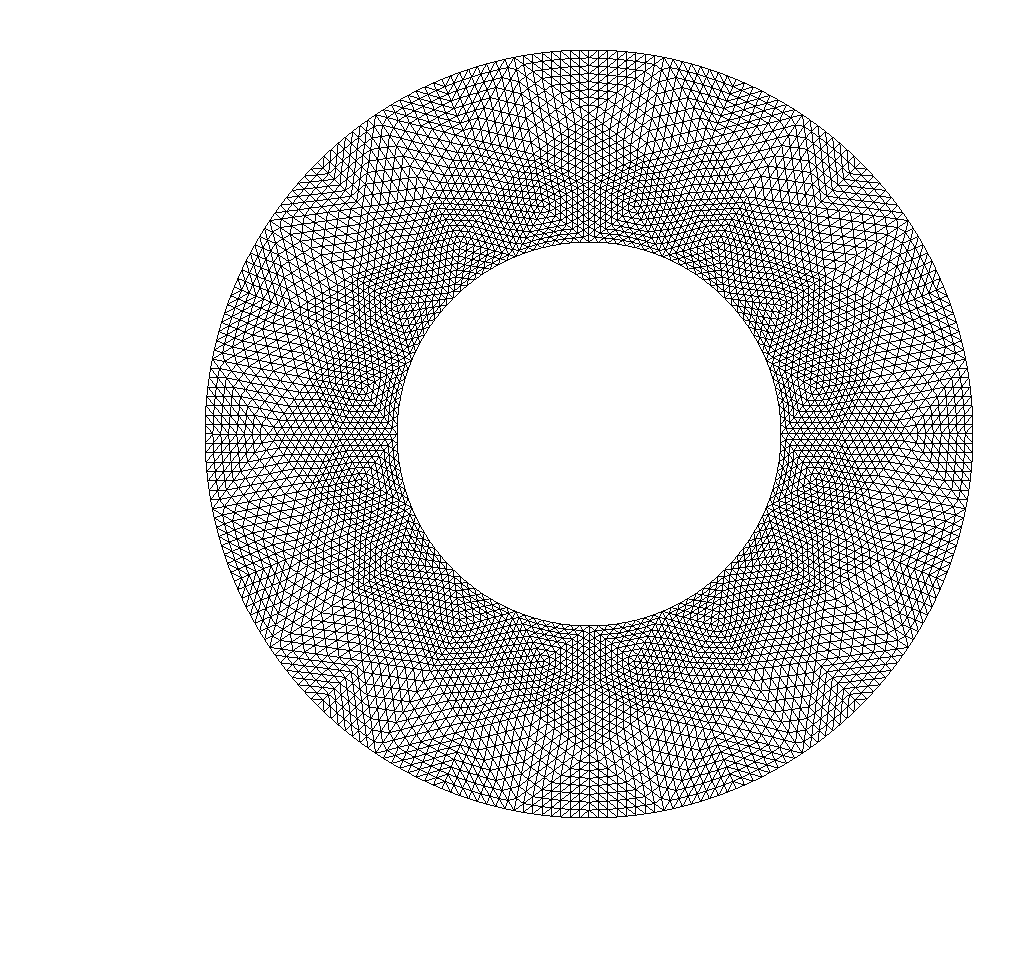}
    \includegraphics[width=0.4\textwidth]{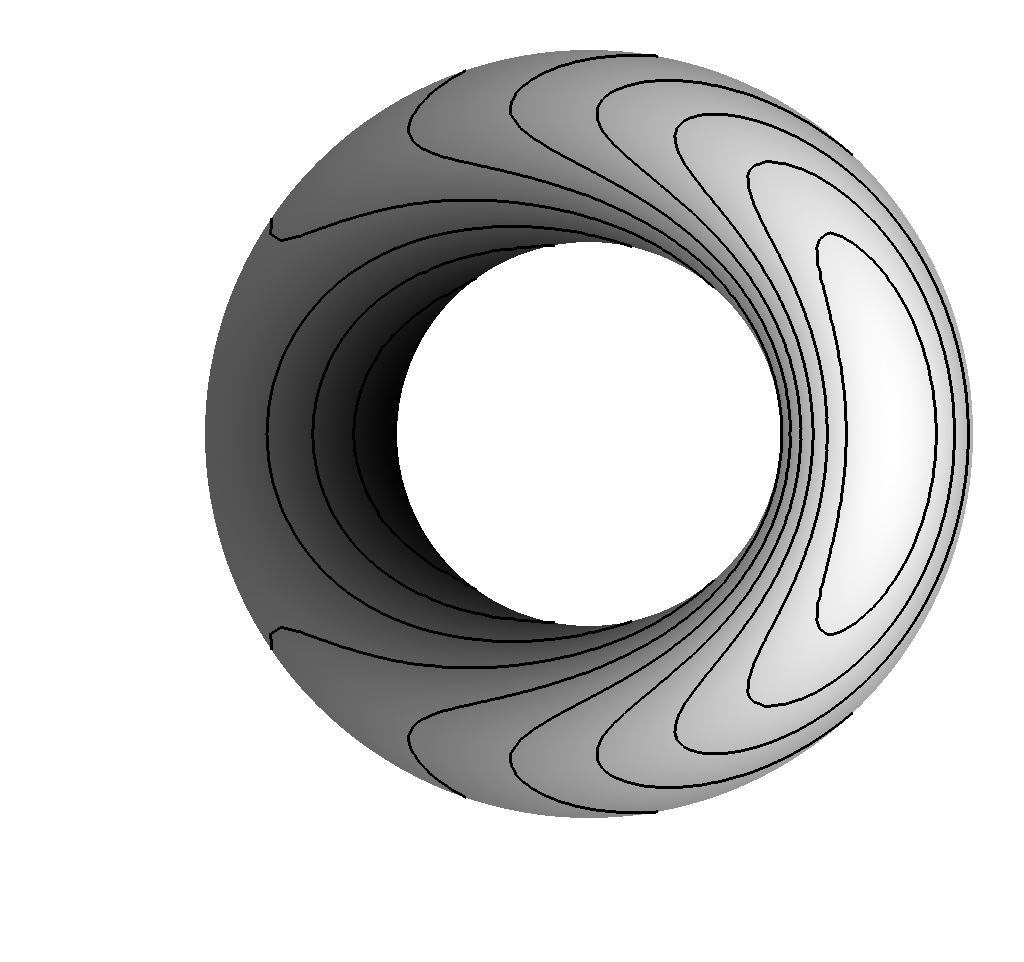}
 \caption{\label{fig:R13_2}  Mesh and  steady state ($t=10$), 3rd order scheme}
 \end{figure}

\section{Conclusion and Outlook}

In this paper, we have demonstrated that a pure continuous Galerkin scheme 
is stable only through the applied boundary conditions. 
No further stabilizations terms are needed. This contradicts the \rev{
erroneous
perception in the hyperbolic community} about pure continuous Galerkin schemes to be unstable without additional stabilizations terms.
In our approach, the application of the SAT technique is essential 
where we impose the boundary conditions weakly. 
Using this approach, we derive a suitable boundary operator from the continuous 
setting to guarantee that the discrete energy inequality is always fulfilled. 
We present a recipe on how these operators can be constructed, in detail,
for scalar two-dimensional problems and two-dimensional systems. 
In numerical experiments, we verify  our theoretical analysis. 
Furthermore, in one test, we  demonstrate the importance of the used 
quadrature rule. The chosen quadrature rule in the numerical integration has to be the same as in the differential operators \revt{ such that the SBP property of our Galerkin scheme is valued}.
 If not,  the Galerkin 
scheme suffers from  stability issues. If stability can be reached only by enforcing the proper dissipative boundary conditions, there is no free meal:
this procedure is very sensitive to any numerical error, like roundoff error or quadrature error.
We think and hope that through our results the
common opinion about continuous Galerkin  and its application 
in CFD problems changes, modulo the restriction we have already described. This result is also interesting from a theoretical point
of view, and emphasizes the positive role that the boundary conditions may have.
In the companion paper \cite{nonlinear}, we  consider and analyze  non-linear (e.g. entropy) stability
of continuous Galerkin schemes. 
Here, the SAT approach will also be important and some approximation for the 
boundary operators will be developed.

\section*{Acknowledgements}
P\"O has been funded by the the SNF grant (Number 200021\_175784) 
and by the UZH Postdoc grant \rev(Number FK-19-104).
This research was initiated by a first visit of JN at UZH,
and really started during ST postdoc at UZH. This postdoc was funded by an SNF grant 200021\_153604. The Los Alamos unlimited release number is LA-UR-19-32410. \\
{P\"O  likes also  to thank Barbara Re (University of Zurich) for  the helpful discussions about the usage of Petsc. 
Finally, all authors  thank  the  anonymous referees whose comments and criticisms have improved the original version of the paper.} \remi{Finally, the authors thanks Professor E. Burman for useful comments.}

\bibliographystyle{unsrt}
\bibliography{literature}

\end{document}